\documentclass[12pt]{amsart} %

\usepackage[latin1]{inputenc}

\usepackage{hyperref}
\hypersetup{colorlinks=true}

\usepackage{color}

\usepackage{amsmath,amsfonts,amsthm,amssymb} 
\usepackage[margin=2cm,bottom=2.5cm,top=2.5cm]{geometry}

\newcommand\N{{\mathbb N}} \newcommand\R{{\mathbb R}}
\newcommand\Z{{\mathbb Z}} \newtheorem{thm}{Theorem}

\newtheorem{rmq}{Remark}[section] \newtheorem{lemma}{Lemma} %
 \newtheorem{prop}{Proposition}

\begin{document}

 \title{Dispersive estimates for the wave and the Klein-Gordon equations in large time inside the Friedlander domain}

\author{Oana Ivanovici}
\address{CNRS $\&$ Sorbonne Université, Laboratoire Jacques-Louis LIONS,
UMR 7598, 4 Place Jussieu, 75252 Paris } \email{oana.ivanovici@sorbonne-universite.fr}
      \thanks{
    \\
 The author was supported by ERC grant ANADEL 757 996.
    }

\begin{abstract} We prove global in time dispersion for the wave and the Klein-Gordon equation inside the Friedlander domain by taking full advantage of the space-time localization of caustics and a precise estimate of the number of waves that may cross at a given, large time. Moreover, we uncover a significant difference between Klein-Gordon and the wave equation in the low frequency, large time regime, where Klein-Gordon exhibits a worse decay that the wave, unlike in the flat space.
\end{abstract}
\maketitle

\section{Introduction and main results}

We consider the following equation on a domain $\Omega$ with smooth boundary,
\begin{equation} \label{KG} 
\left\{ \begin{array}{l} (\partial^2_t-
\Delta+m^2) u(t, x)=0, \;\; x\in \Omega \\ u|_{t=0} = u_0 \; \partial_t
u|_{t=0}=u_1,\\ u|_{\partial_{\Omega}}=0\,.
 \end{array} \right.
 \end{equation}
Here $\Delta$ stands for the Laplace-Beltrami operator on
$\Omega$. If $\partial\Omega\neq \emptyset$, our boundary condition
is the Dirichlet one. Here we will deal with $m\in\{0,1\}$: when $m=0$ we deal with the wave equation, while when $m=1$ we consider the Klein Gordon equation.

If $\Omega=\mathbb{R}^d$ with the flat metric, the solution  $u^m_{\mathbb{R}^d}(t,x)$ to
\eqref{KG} with data $u_0=\delta_{x_{0}}, u_1=0$, $x_0\in \mathbb{R}^d$ has an explicit representation formula 
 \[ 
 u^m_{\mathbb{R}^d}(t,x)=\frac{1}{(2\pi)^d}\int e^{i(x-x_0)\xi}\cos(t\sqrt{|\xi|^2+m^2})\,d\xi\,.
 \]
Fixed time dispersive estimates may be explicitly obtained for any $t\in \R$, with $\chi$ being a smooth cut-off function localizing around $1$: for the wave flow
\begin{gather}\label{dispWErd}
\|\chi(h\sqrt{-\Delta})u^{m=0}_{\mathbb{R}^d}(t,.)\|_{L^{\infty}(\mathbb{R}^d)}\leq
C(d) h^{-d}\min\Big\{1, (h/|t|)^{\frac{d-1}{2}}\Big\},\quad h>0\,,
\end{gather}
and for the Klein-Gordon flow (see \cite{KaOz}, \cite{MaNaTo} and the references therein)
\begin{gather}
\label{dispKGrd}
\|\chi(h\sqrt{-\Delta+1})u^{m=1}_{\mathbb{R}^d}(t,.)\|_{L^{\infty}(\mathbb{R}^d)}\leq C(d) h^{-d}\min\Big\{1, (h/|t|)^{\frac{d-1}{2}}, (h/t)^{\frac{d-1}{2}}\frac{1}{\sqrt{h |t|}}\Big\},\quad h>0\,.
\end{gather}
Alternatively, one may state these estimates for $h\in (0,1)$ (high frequency regime) and state similar estimates with $h=1$ and $\chi$ replaced by $\chi_{0}$, localized around $0$ (low frequency regime). 
On any boundary less Riemannian manifold $(\Omega,g)$ one may follow the same path, replacing the exact formula by a parametrix (which may be constructed locally within a small ball, thanks to finite speed of propagation). However, these techniques usually restrict results to be local in time: depending on other (global) geometrical properties of the underlying manifold, such local in time estimates may be combined with local energy decay estimates to produce global in time Strichartz estimates, provided there is no trapping (or some weak form of it). On a manifold with boundary, picturing light rays becomes much more complicated, and one may no longer think that one is slightly bending flat trajectories. There may be gliding rays (along a convex boundary) or grazing rays (tangential to a convex obstacle) or combinations of both, and as such, even constructing local in time parametrices becomes a difficult task.

Obtaining results for the case of manifolds {\it with} boundary has been surprisingly elusive and the problem had received considerable interest in recent years. Besides harmonic analysis tools, the starting point for these estimates is the knowledge of a parametrix for the linear flow, which turns out to be closely connected to propagation of singularities. It should be noted that parametrices have been available for the boundary value problem for a long time (see \cite{mesj78}, \cite{mesj82}, \cite{mel76}, \cite{meta}, \cite{meta85}, \cite{esk77}) as a crucial tool to establish propagation of singularities for the wave equation on domains. However, while efficient at proving that singularities travel along the (generalized) bicharacteristic flow, they do not seem strong enough to obtain dispersion, as they are not precise enough to capture separation of wave packets traveling with different initial directions.

Inside a strictly convex domain, a parametrix for the wave equation has been constructed in \cite{Annals} (and recently refined in \cite{ILP3}) ; it provides optimal decay estimates, with a $(|t|/h)^{1/4}$ loss compared to \eqref{dispWErd}, uniformly with respect to the distance of the source to the boundary, over a time length of constant size. This involves dealing with an arbitrarily large number of caustics and retain control of their order. In \cite{ILP3}, one considers \eqref{KG} with $m=0$ locally inside the Friedlander domain in dimension $d\geq 2$, defined as the half-space, $\Omega_d=\{(x,y)| x>0, y\in\mathbb{R}^{d-1}\}$ with the metric $g_F$ inherited from the Laplace operator
$\Delta_F=\partial^2_x+(1+x)\Delta_y$. The domain $(\Omega_d,g_F)$ is easily seen to model, locally, a strict convex, as a first order approximation of the unit disk $D(0,1)$ in polar coordinates : set $r=1-x/2$, $\theta=y$. The Friedlander domain is unbounded as we consider $y\in \R^{d-1}$: trajectories escape to spatial infinity, except for those that are shot vertically (no tangential component), and we may expect some sort of long-time estimates. Moreover, in the construction of \cite{ILP3}, the size of the time interval on which the parametrix is constructed does not seem to play a crucial role, unlike in \cite{Annals} where a restriction appears, related to how may wave reflections one may explicitly construct.

In the present work, we indeed obtain a global in time parametrix for waves (following the approach of \cite{ILP3} for large frequencies) and prove global in time optimal dispersive bounds for the linear flow.  We then do the same for the Klein-Gordon equation (for $m=1$), for which, at least in large frequency case, sharp dispersive bounds are equally obtained. The parametrix is obtained (for both $m\in\{0,1\}$) as a sum of wave packets corresponding to the successive reflections on the boundary : as the number of such waves that interfere at a given moment is time depending, for $t$ larger than a certain power of the frequency, the sum of their $L^{\infty}$ norms yields an important loss. In this case, we use the spectral properties of $\Delta_F$ to express the parametrix in terms of its eigenfunctions. This representation of the solution turns out to be particularly useful also in the low frequency case: that situation had not been dealt with in our previous works on the wave equation and has its own difficulties, including some surprising effects in the Klein-Gordon case.
\begin{thm} \label{thmDKG} 
Let $\psi_1\in C^{\infty}_0(\mathbb{R}^+_{*})$. There exists $C>0$ such that, uniformly in $a>0$, $h\in (0,1)$, $t\in \mathbb{R}$, the solution $u^m (t,x,a,y)$ to \eqref{KG} with $m\in\{0,1\}$, with $\Delta$ replaced by $\Delta_F$, $\Omega$ by $\Omega_d$ and with data $(u_0,u_1)=(\delta_{(a,0)},0)$, $\delta_{(a,0)}$ being any Dirac mass at distance $a$ from $\partial\Omega_d$, is such that 
\begin{equation}\label{dispomega}
|\psi_1(h\sqrt{-\Delta_F}))u^m(t,x,a,y)|\leq \frac{C}{h^{d}}\min\Big\{1,\Big(\frac{h}{|t|}\Big)^{\frac{d-2}{2}+\frac 14}\Big\}\,.
\end{equation}
Let $\phi\in C^{\infty}_0((-2,2))$ equal to $1$ on $[0,\frac 32]$. There exists $C_{0}$, such that, uniformly in $a>0$, $t\in\mathbb{R}^*$ 
\begin{equation}\label{dispomegawave}
|\phi(\sqrt{-\Delta_F})u^{m=0}(t,x,a,y)|\leq C_0\min\Big\{1,\frac{1}{|t|^{\frac{d-1}{2}}}\Big\}\,.
\end{equation}
There exist a constant $C_{1}$, such that %
 uniformly in $a>0$, $|t|>1$
\begin{equation}\label{dispomegakg}
|\phi(\sqrt{-\Delta_F})u^{m=1}(t,x,a,y)|\leq C_1\min\Big\{1,\frac{1}{|t|^{\frac{d-2}{2}+\frac 13}}\Big\}\,.
\end{equation}
\end{thm}
Let us comment on these estimates: as far as dispersion is concerned, \eqref{dispomega} is the extension to large times of the main dispersion estimate in \cite{Annals}, and we know it to be optimal (already for small times, due to the presence of swallowtail singularities in the wave front). Note that, unlike in the $\R^{d}$ case where the case $m=1$ improves for $|t|>1/h$, here we obtain the same decay regardless of the presence of a mass. This relates to how we compensate for overlapping waves for large times: even for the wave equation, one has to switch to the gallery modes and use the spectral sum, taking advantage of the fact that gallery modes satisfy a suitably modified wave equation in $\R^{d-1}$, where $-\Delta_{y}$ would be replaced by $-\Delta_{y}+c |\Delta_{y}|^{2/3}$. In estimating time decay, the nonlocal operator plays the role of the mass (it provides more curvature for the characteristic set) and one gets the Klein-Gordon decay in $\R^{d-1}$. Hence, when adding a mass term, one gets $-\Delta_{y}+c |\Delta_{y}|^{2/3}+1$ and one does not get additional decay.

For low frequencies, \eqref{dispomegawave} tells us that the long time dispersion for the wave equation is the same as in $\R^{d}$. This actually holds on more generic non trapping exterior domains. However, \eqref{dispomegakg} exhibits a loss compared to \eqref{dispomegawave}, and this new effect appears to be strongly connected to our domain. In fact, both this loss and the absence of gain for $|t|>1/h$ in the large frequency range may be informally related to the same operator that was introduced above: $-\Delta_{y}+c |\Delta_{y}|^{2/3}+1$. If one restricts to the 2D case, e.g. $y\in \R$, then one has to deal with the following oscillatory integral, where $c>0$,
$$
v(z,t)=\int e^{i t (z\eta -\sqrt{|\eta|^{2}+c |\eta|^{4/3}+m^2})} \psi_{1}(\eta)\,d\eta\,,
  $$
  for which one may check that there is at least a degenerate critical point of order two if $m=1$ while there is none if $m=0$. In fact, the second derivative of the phase does not depend on $z$, it does have a zero $\eta_{0}$ and then one may chose $z$ so that the first derivative vanishes as well for that $\eta_{0}$. Interestingly enough, this effect occurs around $\eta\sim 1$ (and with $c\sim \omega_{1}$, the first zero of $\mathrm{Ai}(-\cdot)$), e.g. for a moderately transverse direction, and not because of a very vertical direction. As such, one may see this effect arise in more realistic wave guide type  domains and we believe it to be of interest.

Finally, one can immediately deduce a suitable set of (long time) Strichartz estimates in our setting, which are exactly the long time version of the ones stated in \cite{Annals}, for both wave and Klein-Gordon equations.

In the remaining of the paper, $A\lesssim B$ means that there exists a constant $C$ such that $A\leq CB$ and this constant may change from line to line but is independent of all parameters. It will be explicit when (very occasionally) needed. Similarly, $A\sim B$ means both $A\lesssim B$ and $B\lesssim A$.

\section{The half-wave propagator: spectral analysis and parametrix construction}
This part follows closely \cite[Section 2]{ILP3} : we denote $Ai$ the standard Airy function and we define
\begin{equation}
  \label{eq:Apm}
  A_\pm(z)=e^{\mp i\pi/3} Ai(e^{\mp i\pi/3} z)=\Psi(e^{\mp i\pi/3} z)e^{\mp\frac 23 i z^{3/2} }\,,\,\,\text{ for } \,
  z\in \mathbb{C}\,,
\end{equation}
then one checks that $Ai(-z)=A_+(z)+A_-(z)$. 
We have $\Psi(z)\simeq z^{-1/4}\sum_{j=0}^{\infty} a_j z^{-3j/2}$,  $a_0=\frac{1}{4\pi^{3/2}}$.
\begin{lemma}\label{lemL}(see \cite[Lemma 1]{ilpCE})
Define 
 $ L(\omega)=\pi+i\log \frac{A_-(\omega)}{A_+(\omega)} \,,\,\,\text{
    for }\, \omega \in \R\,$, 
then $L$ is real analytic and strictly increasing. We also
have
\begin{equation}
  \label{eq:propL}
  L(0)=\pi/3\,,\,\,\lim_{\omega\rightarrow -\infty} L(\omega)=0\,,\,\,
  L(\omega)=\frac 4 3 \omega^{\frac 3 2}+\frac{\pi}{2}-B(\omega^{\frac 3
    2})\,,\,\,\text{ for } \,\omega\geq 1\,,
\end{equation}
with $B(u)\simeq \sum_{k=1}^\infty b_k u^{-k}$, $b_k\in\R$ and $b_1=\frac{5}{4\times 2^2}> 0\,$.
Moreover, $L'(\omega_k)=2\pi \int_0^\infty Ai^2(x-\omega_k) \,dx\,$,
where here and thereafter, $\{-\omega_k\}_{k\geq 1}$ denote the zeros of the Airy function in decreasing order.
\end{lemma}

We briefly recall a Poisson summation formula that will be crucial to construct a parametrix :  \begin{lemma} (see \cite{ilpCE})
  In $\mathcal{D}'(\R_\omega)$, one has
       $ \sum_{N\in \Z} e^{-i NL(\omega)}= 2\pi \sum_{k\in \N^*} \frac 1
    {L'(\omega_k)} \delta(\omega-\omega_k)\,$.
Hence, for $\phi(\omega)\in C_{0}^{\infty}$, 
  \begin{equation}
    \label{eq:AiryPoissonBis}
        \sum_{N\in \Z} \int e^{-i NL(\omega)} \phi(\omega)\,d\omega = 2\pi \sum_{k\in \N^*} \frac 1
    {L'(\omega_k)} \phi(\omega_k)\,.
  \end{equation}
\end{lemma}

\subsection{Spectral analysis of the Friedlander model}\label{sectspectralloc}

Let $\Omega_d$ be the half-space $\{(x,y)\in\mathbb{R}^d|, x>0,y\in\mathbb{R}^{d-1}\}$ and consider the operator
$\Delta_F=\partial^{2}_{x}+(1+x)\Delta_{y}$ on $\Omega_d$ with Dirichlet boundary condition. After a Fourier transform in the $y$ variable, the operator $-\Delta_{F}$ becomes
$-\partial^2_x+(1+x)|\theta|^2$. For any $|\theta|\neq 0$, this is a positive self-adjoint operator 
on $L^2(\mathbb{R}_+)$, with compact resolvent. 

\begin{lemma}\label{lemorthog} (\cite[Lemma 2]{ilpCE})
There exist orthonormal eigenfunctions $\{e_k(x,\theta)\}_{k\geq 0}$ with their corresponding eigenvalues $\lambda_k(\theta)=|\theta|^2+\omega_k|\theta|^{4/3}$, which form a Hilbert basis of $L^{2}(\mathbb{R}_{+})$.  These eigenfunctions have an explicit form
\begin{equation}\label{eig_k}
 e_k(x,\theta)=\frac{\sqrt{2\pi}|\theta|^{1/3}}{\sqrt{L'(\omega_k)}}
Ai\Big(|\theta|^{2/3}x-\omega_k\Big),
\end{equation}
where $L'(\omega_k)$ is given in Lemma \eqref{lemL}, which yields
$\|e_k(.,\theta)\|_{L^2(\mathbb{R}_+)}=1$. 
\end{lemma}
In a classical way, for $x_0>0$, the Dirac distribution $\delta_{x=x_0}$ on $\mathbb{R}_+$ may be decomposed in terms of eigenfunctions $\{e_k\}_{k\geq 1}$ as follows
\begin{equation}\label{defdelta}
 \delta_{x=x_0}=\sum_{k\geq 1} e_k(x,\theta)e_k(x_0,\theta).
\end{equation}
This allows to obtain (at least formally) the Green function associated to the half-wave propagator for \eqref{KG} in $(0,\infty)\times \Omega_d$:
\begin{equation}\label{greenfct} 
G^{m,\pm}((t,x,y),(t_0,x_0,y_0))=\sum_{k\geq 1}
\int_{\mathbb{R}^{d-1}}e^{\pm i(t-t_0)\sqrt{m^2+\lambda_k(\theta)}}
e^{i<(y-y_0),\theta>}
e_k(x,\theta)e_k(x_0,\theta)d\theta\,.
\end{equation}
In the following we fix $+$ sign and write $G^{0}$ for the wave and $G^{1}$ for Klein-Gordon Green function.  \\

We will deal separately with the following situations : \\
\begin{itemize}
\item The high frequency case "$\sqrt{-\Delta_F}\sim 2^{j}$" for $j\in \mathbb{N}$, $j\geq 1$ ; this corresponds to $\sqrt{\lambda_k(\theta)}\sim 2^j$. The main situation is the "tangent" case $|\theta|\sim 2^j$,
which corresponds to tangent directions, when the number of reflections on the boundary is at its highest. Indeed, denoting $(\xi,\theta)$ the dual variables of $(x,y)$ in the cotangent space, we have 
\begin{equation}\label{lambda_k}
\lambda_k(\theta)=|\theta|^2+\omega_k|\theta|^{4/3},\quad \xi^2\simeq \omega_k|\theta|^{4/3}
\end{equation}
and when $\lambda_k(\theta)\sim 2^{2j}$ and $|\theta|^2\sim 2^{2j}$ then $(\xi/|\theta|)^2$ may be small and we deal with a large number of reflecting rays on the boundary and their limits, the gliding rays (the case when $(\xi/|\theta|)^2$ is bounded from below by a fixed constant corresponds to transverse rays, which may also reflect many times on the boundary but provide much better estimates). This is the (only) case dealt with in \cite{Annals}, \cite{ILP3} for the {\it local in time} wave equation in the Friedlander domain. Here we start by recalling the main steps to the parametrix construction from \cite{ILP3} and its main properties, and extend it to all $t$; we further adapt this construction to the Klein-Gordon case and proceed with fixed time decay bounds (first for the "tangent" part).

Let $\psi_{1}\in  C^{\infty}_0([\frac 12,\frac 32])$ be a smooth function valued in $[0,1]$ and equal to $1$ near $1$. As $
-\Delta_{F} ( e^{i<y,\theta>} e_{k}(x,\theta)) =\lambda_{k}(\theta) e^{i<y,\theta>} e_{k}(x,\theta)\,$, we  introduce the spectral cut-off $\psi_{1}(h\sqrt{-\Delta_F})$ in $G^m$, where $h\in (0,1/2]$ is a small parameter. %
Let $\psi\in C^{\infty}_0([\frac 12,\frac 32])$ such that $\psi_1(\xi)+\sum_{j\geq 1}\psi(2^j\xi)=1$. We define the "tangent" part of the Green function $G^m_h$ with $(x_0,y_0)=(a,0)$ as follows
\begin{equation}\label{greenfctHF} 
G^{\#,m}_{h}(t,x,a,y):=\sum_{k\geq 1}
\int_{\mathbb{R}^{d-1}}e^{i t\sqrt{m^2+\lambda_k(\theta)}}
e^{i<y,\theta>}
 \psi_1(h|\theta|)\psi_{1}(h\sqrt{\lambda_{k}(\theta)})
 e_k(x,\theta)e_k(a,\theta)d\theta\,.
\end{equation}
As remarked in \cite{ILP3}, the significant part of the sum over $k$ in \eqref{greenfctHF} becomes then a finite sum over $k\lesssim 1/h$, considering the asymptotic expansion of $\omega_{k}\sim k^{2/3}$ (and corresponds to initial angles $(\xi/|\theta|)^2\sim (\omega_k|\theta|^{-2/3})\sim \omega_kh^{2/3}\lesssim 1$, where the last inequality is due to \eqref{lambda_k} and the spectral cut-offs). Reducing the sum to $k\lesssim 1/h$ is equivalent to adding a spectral cut-off $\phi(x+h^2D_{x}^{2}/|\theta|^{2/3})$ in $G^{\#,m}_h$, where $\phi\in C^{\infty}_0((-2,2))$: as the operator
$-\partial^2_{x}+x|\theta|^2$ has the same eigenfunctions  $e_k(x,\theta)$ associated to the eigenvalues $\lambda_k(\theta)-|\theta|^2=\omega_k |\theta|^{4/3}$, then $(x+h^2D_{x}^{2}/|\theta|^2)e_{k}(x,\theta)=(\omega_k |\theta|^{4/3}/|\theta|^2) e_{k}(x,\theta)$ and this new localization operator is exactly associated by symbolic calculus to the cut-off $\phi(\omega_{k}/ |\theta|^{2/3})$, that we can introduce in \eqref{greenfctHF} without changing the contribution modulo $O(h^{\infty})$ terms.

It will be convenient to introduce a new, small parameter $\gamma$ satisfying $\sup{(a,h^{2/3})}\lesssim \gamma\lesssim 1$ and then split the (tangential part $G^{\#,m}_h$ of the) Green function into a dyadic sum $G^m_{h,\gamma}$ corresponding to a dyadic partition of unity supported for $\omega_k /|\theta|^{2/3}\sim \gamma \sim 2^j\sup{(a,h^{2/3})}\lesssim 1$.
Let $\psi_2(\rho):=\phi(\rho)-\phi(2\rho)$, then $\psi_2\in C^{\infty}_0([\frac 34,2])$ is equal to $1$ on $[1,\frac 32]$ and decompose $\phi$ as follows
\begin{equation}\label{partunitpsi2}
\phi(\cdot) %
=\phi_{\sup{(a,h^{2/3})}}(\cdot)+\sum_{\gamma=2^j \sup{(a,h^{2/3})}, 1\leq j<\log_2(1/\sup{(a,h^{2/3})})}\psi_2(\cdot/\gamma),
\end{equation}
which allows to write $G^{\#,m}_{h}=\sum_{\sup{(a,h^{2/3})}\leq \gamma\lesssim 1}G^m_{h,\gamma}$ where (rescaling the $\theta$ variable for later convenience) $G^m_{h,\gamma}$ takes the form
\begin{multline}\label{greenfctbis}
G^{m}_{h,\gamma}(t,x,a,y)  =  \sum_{k\geq 1}
\frac{1}{h^{d-1}} \int_{\mathbb{R}^{d-1}}e^{\pm it\sqrt{m^2+\lambda_k(\eta/h)}}
e^{\frac ih<y,\eta>}  \psi_1(|\eta|)\psi_{1}(h\sqrt{\lambda_{k}(\eta/h)})\\
{}\times\psi_{2}(h^{2/3}\omega_{k}/(\eta^{2/3}\gamma)) e_k(x,\eta/h)e_k(a,\eta/h)d\eta\,.
\end{multline}
\begin{rmq}
When $\gamma=\sup{(a,h^{2/3})}$, according to \eqref{partunitpsi2}, we should write $\phi_{\sup{(a,h^{2/3})}}$ instead of $\psi_2(\cdot/\sup{(a,h^{2/3})})$ in \eqref{greenfctbis}. However, for values $h^{2/3}\omega_k\lesssim \frac 12 \sup{(a,h^{2/3})}$, the corresponding Airy factors are exponentially decreasing and yield an irrelevant contribution; therefore writing $\phi_{\sup{(a,h^{2/3})}}$ or $\psi_2(\cdot/\sup{(a,h^{2/3})})$ yields the same contribution in $G^m_{h,\sup{(a,h^{2/3})}}$ modulo $O(h^{\infty})$. In fact, when $a<h^{2/3}$ is small, there are no $\omega_k$ satisfying $h^{2/3}\omega_k/|\eta|^{2/3}<\frac 34 h^{2/3}$ as $\omega_k\geq \omega_1\simeq 2.33$ and $|\eta|\in [\frac 12, \frac 32]$; on the other hand, when $a\gtrsim h^{2/3}$ and $h^{2/3}\omega_k/|\eta|^{2/3}\leq a/2$ then the Airy factor of $e_k(a,\eta/h)$ is exponentially decreasing (see \cite[Section 2.1.4.3]{AFbook}). In order to streamline notations, we use the same formula \eqref{greenfctbis} for each $G^m_{h,\gamma}$.  
\end{rmq}

\item When "$\sqrt{-\Delta_F}\sim 2^j$" and $0< |\theta|\leq 2^{j-1}$ then $\omega_k|\theta|^{4/3}\sim 2^{2j}$. According to \eqref{lambda_k}, this case corresponds to purely transverse directions (as $(\xi/|\theta|)^2\sim \omega_k/|\theta|^{2/3}\sim 2^{2j}/|\theta|^2\geq 4$) and hadn't been dealt with in the previous works \cite{Annals}, \cite{ILP3} (as, in bounded time, the situation $\xi^2\gtrsim 1$ corresponds to a bounded number of reflections and follows from \cite{blsmso08}). We define the "transverse" part of the Green function $G^m$ as $G^{\flat,m}_{h}:=\sum_{j\geq 1} G^{m}_{h,2^jh}$, where
\begin{equation}
\label{greenfctHFsmallertheta} 
G^{m}_{(h,\tilde h)}(t,x,a,y):=\sum_{k\geq 1}
\int_{\mathbb{R}^{d-1}}e^{\pm i t\sqrt{m^2+\lambda_k(\theta)}}
e^{i<y,\theta>}
 \psi(\tilde h|\theta|)\psi_{1}(h\sqrt{\lambda_{k}(\theta)})e_k(x,\theta)e_k(a,\theta)d\theta\,.
\end{equation}
Notice that in \eqref{greenfctHFsmallertheta} the sum over $k$ is finite as on the support of $\psi_1$ we have $\omega_k \lesssim \tilde h^{4/3}/h^2$. 
In this case we may consider large initial values $1\lesssim a\leq  4(\frac{\tilde h}{h})^2$ (as, for $a\geq 4(\frac{\tilde h}{h})^2$ the factor $e_k(a,\theta)$ yields $G^{m}_{(h,\tilde h)}(t,x,a,y)=O((\frac{\tilde h^{4/3}}{h^2})^{-\infty})=O(h^{\infty})$). After several suitable changes of variable, obtaining dispersive bounds for $G^m_{(h,\tilde h)}$ will reduce to uniform bounds for $G^m_{h^3/\tilde h^2,\gamma=1}$ (as in \eqref{greenfctbis} with $h^3/\tilde h^2$ instead of $h$ and $\gamma=1$). In some suitable sense, we reduce this case to the previous one by rescaling, which is consistent with the geometric picture of light rays: transverse rays for long time look the same (after dezooming) as tangential rays on a short time interval.

\item In the case of small frequencies "$|\sqrt{-\Delta_F}|\leq 2$ ", "$|\sqrt{-\Delta_y}|\leq 2$" we let $\phi\in C^{\infty}_0((-2,2))$ such that $\phi=1$ on $[0,\frac 32]$ and introduce the spectral cut-off $\phi(\sqrt{-\Delta_F})$ in $G^m$. We have to consider all possible situations $|\theta|\simeq 2^{-j}$ for $j\in \mathbb{N}$. As $(1-\phi(|\theta|))$ is supported for $|\theta|\geq 3/2$, then $(1-\phi(|\theta|))\phi(\sqrt{ \lambda_k(\theta)})=0$, and we can add $\phi(|\theta|)$ into the symbol of $\phi(\sqrt{-\Delta_F})G^m$. Let $\psi_2\in C^{\infty}_0([\frac 34,2])$ such that $\psi_2(\rho):=\phi(\rho)-\phi(2\rho)$ then $\sum_{j\in\mathbb{N}} \psi_2(2^j|\theta|)=\phi(|\theta|)$ and on the support of $\psi_2(2^j|\theta|)$ we have $|\theta|\sim 2^{-j}$ which allows to use again Lemma \ref{lemorthog} and \eqref{greenfct}. This situation hasn't been encountered in our previous works. Let $a>0$ and 
\begin{equation}
  \label{eq:uhterWGMSF}
  G^m_{j}(t,x,a,y):=\sum_{k\geq 1} \int e^{i \Big(y\theta+t\sqrt{m^2+\lambda_k(\theta)}\Big)}   \psi(2^{j}|\theta|)\phi(\sqrt{\lambda_k(\theta)})e_k(x,\theta)e_k(a,\theta)d\theta,
\end{equation}
then $\sum_j G^m_{j}$ is the solution to \eqref{KG} with data $(\phi(\sqrt{-\Delta_F})\delta_{(a,0)},0)$. 
In this case we will see new effects arise in the case of the Klein-Gordon equation, as mentioned in the introduction.
\end{itemize}
\subsubsection*{The spectral sum $G^{\#,m}_h$ in terms of reflected waves }\label{secparamHF} 

Using the Airy-Poisson formula \eqref{eq:AiryPoissonBis}, we obtain a parametrix, both as a "spectral" sum and its counterpart after Poisson summation. Let $G^m_{h,\gamma}$ as in \eqref{greenfctbis}, then, using \eqref{eq:AiryPoissonBis}, its Fourier transform in $y$, that we denote $\hat{G}^m_{h,\gamma}$, equals
\begin{multline}
\hat{G}^m_{h,\gamma}(t,x,a,\eta/h)= \frac 1 {2\pi} \sum_{N\in \Z} \int_{\R}e^{-i NL(\omega)} e^{i\frac t h |\eta| \sqrt{1+\omega (h/|\eta|)^{2/3}+m^2(h/|\eta|)^2}}
\chi_{1}(\omega)\psi_1(|\eta|)\\ \psi_{1}(|\eta|\sqrt{1+(h/|\eta|)^{2/3}\omega})\psi_{2}((h/|\eta|)^{2/3}\omega/\gamma)
{}\times \frac{|\eta|^{2/3}}{h^{2/3}}  Ai (x|\eta|^{2/3}/h^{2/3} -\omega) Ai(a|\eta|^{2/3}/h^{2/3}-\omega) d\omega\,.
\end{multline}
Here, $\chi_{1}(\omega)=1$ for $\omega>2$ and $\chi_{1}(\omega)=0$ for $\omega<1$, and obviously $\chi_{1}(\omega_{k})=1$ for all $k$, as $\omega_{1}>2$. At this point, as $\eta\in [\frac 12,\frac 32]$, we may drop also the $\psi_{1}(|\eta|\sqrt{1+(h/|\eta|)^{2/3}\omega})$ localization by support considerations (slightly changing any cut-off support if necessary).  
Recall that
\begin{equation}
  \label{eq:bis47}
  Ai(x|\eta|^{2/3}/h^{2/3}-\omega)=\frac{(|\eta|/h)^{1/3} }{2\pi} \int e^{\frac i h |\eta| (\frac{\sigma^{3}}{3}+\sigma(x-(h/|\eta|)^{2/3}\omega))} \,d\sigma\,.
\end{equation}
Rescaling $\alpha=(h/|\eta|)^{2/3} \omega$ yields
\begin{multline}\label{uhgamN}
    G^m_{h,\gamma}(t,x,a,y):=\frac 1 {(2\pi )^{3}h^{d+1}} \sum_{N\in \Z} \int_{\R^{d-1}}\int_{\R}\int_{\R^{2}} e^{\frac i h \Phi^m_{N,h,a}(t,x,y,\sigma,s,\alpha,\eta)}  |\eta|^{2} \psi_1(|\eta|)
\psi_{2}(\alpha/\gamma) \, ds d\sigma d\alpha d\eta\,,
\end{multline}
where we have set $\Phi^m_{N,a,h}=\Phi^m_{N,a,h}(t,x,y,\sigma,s,\alpha,\eta)$ with
\begin{equation}
  \label{eq:bis49}
    \Phi^m_{N,h,a}=y\eta+|\eta|\Big(\frac{\sigma^{3}} 3+\sigma(x-\alpha)+\frac {s^{3}} 3+s(a-\alpha)-N\frac{h}{|\eta|} L(|\eta|^{2/3} \alpha/h^{2/3})+t \sqrt{1+\alpha+m^2(h/|\eta|)^2}\Big)\,.
\end{equation}

\begin{rmq}
The critical points with respect to $s,\sigma$ of \eqref{eq:bis49} satisfy $\sigma^2=\alpha-x$, $s^2=\alpha-a$ and on the support of $\psi_2$ we have $\alpha\simeq \gamma$. Making the change of coordinates $\alpha=\gamma A$, $s=\sqrt{\gamma}S$, $\sigma=\sqrt{\gamma}\Upsilon$ (see Section \ref{sectang} for $\gamma\simeq a$ and Section \ref{sectransv} for $\gamma>8a$) transforms $G^m_{h,\gamma}$ into an integral with parameter $\lambda_{\gamma}:=\gamma^{3/2}/h$ : in order to apply stationary phase arguments, this parameter needs to be larger than a power $h^{-\varepsilon}$ for some $\varepsilon>0$ and therefore the parametrix \eqref{uhgamN} is useful only when $\gamma\gtrsim h^{2(1-\varepsilon)/3}$. If $a>h^{2(1-\varepsilon)/3}$ for some $\varepsilon>0$ this will always be the case, but when $a\leq h^{2(1-\varepsilon)/3}$ and $\max\{a,h^{2/3}\}\lesssim \gamma\leq h^{2(1-\varepsilon)/3}$ we cannot use \eqref{uhgamN} anymore. 
\end{rmq}
Before starting the proof of Theorem \ref{thmDKG} in the high frequency case, we show that this reduces to the two dimensional case. We first need to establish a propagation of singularities type result for $G^m_h$, that will be necessary in order to obtain dispersion estimates in the $d-2$ tangential variables. 
\begin{lemma}
\label{lemPofS}
Let $ G^m_h(t,x,a,y)= \sum_{\max\{h^{2/3},a\}\leq \gamma<1}G^m_{h,\gamma}(t,x,a,y)$. There exists $c_{0}$ such that
  \begin{equation}
    \label{eq:2bis}
  \sup_{x,y,t\in \mathcal{B}}  |  G^m_{h}(t,x,a,y)|\leq C h^{-d}  O(( h/ |t|)^{\infty})\,,\quad \forall |t|>h,
  \end{equation}
where $\mathcal{B}=\{ 0\leq x\lesssim a, |y|\leq c_{0} t, 0<h\leq |t| \}$.
\end{lemma}

\begin{proof}
A proof of this Lemma has been given in \cite[Lemma 3.2]{ILLP} in the case of the wave equation on a generic strictly convex domain, based on propagation of singularities type results. In \cite{ILLP}, the time is restricted to a bounded interval $h\leq |t|\leq T_0$ for some small $T_0$, which is the time interval considered in that paper. Same arguments apply in the case of the wave of Klein-Gordon equation in the Friedlander model domain, where the time can be taken large. 
\end{proof}

Using Lemma \ref{lemPofS} and the fact that our model domain is isotropic, we can integrate in the $d-2$ tangential variables $\eta/|\eta|$ and reduce the analysis to the two-dimensional case (by rotational invariance). As such, in the rest of the paper, as long as we deal with the high frequency case, we consider $d=2$. In the next two sections we consider only high frequencies. 

\section{The parametrix regime in 2D. The high frequency case }
\label{sec:parametrix-regime}
\subsection{Localizing waves for $\max\{a,h^{2/3(1-\varepsilon)}\}\lesssim \gamma \lesssim 1$}
\label{sec:local-waves}

In \cite{Annals}, it has been shown that, as long as $a\gtrsim h^{4/7}=h^{2(1-1/7)/3}$, only a finite number of integrals in $G^{m=0}_{h,a}$ may overlap ; for such $a$ and for $\gamma\gtrsim a$, it follows that at a fixed time $t$, the supremum of the sum in \eqref{uhgamN} is essentially bounded by the supremum of a finite number of waves that live at time $t$.  
Later on, in \cite{ILP3}, it has been shown that, when $h^{2/3(1-\epsilon)}\leq a\lesssim \gamma \leq h^{4/7}$, the number of waves that cross each other at a given $t$ in the sum \eqref{uhgamN} becomes unbounded even for small $t\lesssim 1$. Moreover, the number of $N$ with "significant contributions" had been (sharply) estimated which allowed to obtain refined bounds in this regime (better than in \cite{Annals}, where only the "spectral" version \eqref{greenfctbis} of $G^{m=0}_{h,a}$ was then available for small $a$). 
 
We claim that, although an important contribution comes from $m=1$, it doesn't exceed the one already obtained for $m=0$ : as a consequence, for both $m\in\{0,1\}$, we have to sum-up exactly the same number of terms in \eqref{uhgamN}. 
Assume (without loss of generality) $t>0$. Let $m\in\{0,1\}$, $d=2$, $\eta\in \mathbb{R}$, then $G^m_{h,\gamma}=\sum_N V^m_{N,\gamma}$ where we have set
\begin{equation}\label{defVNgamma}
V^m_{N,\gamma}(t,x,a,y):=\frac{1}{(2\pi)^3h^{3}}\int_{\R^{2}}\int_{\R^{2}} e^{\frac ih  \Phi^m_{N,a,h}(t,x,y,\sigma,s,\alpha,\eta)}  \eta^{2} \psi({\eta})  \psi_{2}(\alpha/\gamma) \, ds d\sigma  d\alpha d\eta\,.
\end{equation}

\begin{lemma}\label{lemmeNtpetit}(see \cite[Lemma 4]{ILP3})
At fixed $t>\sqrt{\gamma}$, $\sum_{\sqrt \gamma N \not\in O(t)} V^m_{N,\gamma}(t,\cdot)=O(h^{\infty})$.
\end{lemma}
The proof for $m=0$ had been given in \cite{ILP3} and is based on non-stationary phase arguments; the case $m=1$ is dealt with in the same way (as $m^2$ comes with a very small factor $h^2$). Moreover, the proof only uses the variable $T:=t/\sqrt{\gamma}$, and not the size of $t$ which can be large. 

Next, we estimate the number of overlapping waves when $m=1$ and prove that there is no significant difference with respect to $m=0$. In order to do that, we have to introduce some notations.
For a given space-time location $(t,x,y)$, let $\mathcal{N}^m_{\gamma}(t,x,y)$ be the set of $N$ with significant contributions in \eqref{uhgamN} (e.g. for which there exists a stationary point for the phase in all variables), and let $\mathcal{N}^m_{1,\gamma}(t,x,y)$ be the set of $N$ belonging to $\mathcal{N}^m_{\gamma}(t',x',y')$ for some $(t',x',y')$ sufficiently close to $(t,x,y)$ such that $|t'-t|\leq \sqrt{\gamma}$, $|x-x'|< \gamma$ and $|y'+t'\sqrt{1+\gamma}-y-t\sqrt{1+\gamma}|<\gamma^{3/2}$,
 \begin{gather}\nonumber
\mathcal{N}^m_{\gamma}(t,x,y)=\{N\in\mathbb{Z}, (\exists)(\sigma,s,\alpha,\eta) \text{ such that } \nabla_{(\sigma,s,\alpha,\eta)}\Phi^m_{N,a,h}(t,x,y,\sigma,s,\alpha,\eta)=0 \},\\
\nonumber
\mathcal{N}^m_{1,\gamma}(t,x,y)=\cup_{\{(t',x',y')||t'-t|\leq \sqrt{\gamma}, |x-x'|< \gamma, |y'+t'\sqrt{1+\gamma}-y-t\sqrt{1+\gamma}|<\gamma^{3/2}\}}\mathcal{N}^m_{\gamma}(t',x',y').
 \end{gather}

\begin{prop}\label{propcardN}
Let $t>h$ and $1\gtrsim\gamma\gtrsim h^{2(1-\varepsilon)/3}$. The following estimates hold true:
\begin{itemize}
\item We control the cardinal of $\mathcal{N}^m_{1,\gamma}(t,x,y)$,
\begin{equation}
  \label{eq:113}
  \left| \mathcal{N}^m_{1,\gamma}(t,x,y)\right| \lesssim O(1)+\frac{t}{\gamma^{1/2} (\gamma^{3}/h^{2})}+m^2\frac{|t|h^2}{\gamma^{3/2}}\,,
\end{equation}
and this bound is optimal.
\item The contribution of the sum over $N\notin \mathcal{N}^m_{1,\gamma}(t,x,y)$ in \eqref{uhgamN} is $O(h^{\infty})$.
\end{itemize}
\end{prop}
\begin{rmq}
In the present work we adapt the proof from \cite{ILP3} to all time. For $m=0$, the sharp bounds \eqref{eq:113} are obtained in \cite[Prop.1]{ILP3} by estimating the distance between any two points $N_1,N_2\in \mathcal{N}^{m=0}_{1,\gamma}(t,x,y)$ and the proof holds for any $t$. When $m=1$, the difference comes from the term $m^2(h/\eta)^2$ in the coefficient of $t$ which may contribute significantly only when $th \gg 1$, a case that hadn't be dealt with in our previous works. 
Following closely the approach in \cite{ILP3}, we notice that the terms involving $m^2(h/\eta)^2$ give rise to the last addendum $m^2\frac{|t|h^2}{\gamma^{3/2}}$ in \eqref{eq:113}. However, as $|t|h^{2}/\gamma^{3/2}\leq t h^2/\gamma^{7/2}$, for $\gamma\lesssim 1$ there is no real difference between the cases $m=0$ and $m=1$.
\end{rmq}

\begin{proof}(of Proposition \ref{propcardN})
We sketch the proof of \eqref{eq:113} in the case $m=1$. 
Let
\begin{equation}
  \label{eq:83ff}
  x=\gamma X\,, \alpha=\gamma A\,, t=\sqrt \gamma T\,, s=\sqrt \gamma S\,, \sigma=\sqrt \gamma \Upsilon\,, y+t\sqrt{1+\gamma}=\gamma^{3/2} Y\,,
\end{equation}
\begin{multline}\label{PhiNagammadef}
  \Psi^m_{N,a,\gamma,h}(T,X,Y, \Upsilon,S,A,\eta): =\eta\left(Y+ \Upsilon^3/3+ \Upsilon(X-A)+S^3/3+S(\frac{a}{\gamma}-A)\right.\\
    {} +T\frac{\sqrt{1+\gamma A+m^2h^2/\eta^2}-\sqrt{1+\gamma}}{\gamma}
    {}-\left.\frac 43 NA^{3/2}\right)+\frac{N}{\lambda_{\gamma}}B(\eta\lambda_{\gamma}A^{3/2}).
\end{multline}
then
$\gamma^{3/2}\Psi^m_{N,a,\gamma,h}(T,X,Y, \Upsilon,S,A,\eta)=\Phi^m_{N,a,h}(\sqrt{\gamma}T,\gamma X,\gamma^{3/2}Y-\sqrt{\gamma}\sqrt{1+\gamma}T,\sqrt{\gamma} \Upsilon,\sqrt{\gamma}S,\gamma A,\eta)$,
and, in the new variables, the phase function in \eqref{defVNgamma} becomes $\lambda_{\gamma}\Psi^m_{N,a,\gamma,h}$ where $\lambda_{\gamma}=\gamma^{3/2}/h$. 
The critical points of $\Psi^m_{N,a,\gamma,h}$ with respect to $ \Upsilon,S, A,\eta$ satisfy
\begin{gather}\label{critSig}
 \Upsilon^2+X=A,\quad S^2+a/\gamma=A,\\
\label{critA}
T=2\sqrt{1+\gamma A+m^2h^2/\eta^2}\Big( \Upsilon+S+2N\sqrt{A}(1-\frac 34 B'(\eta\lambda_{\gamma}A^{3/2}))\Big),\\
\label{criteta}
Y+T\Big(\frac{\sqrt{1+\gamma A+m^2h^2/\eta^2}-\sqrt{1+\gamma}}{\gamma}-\frac{m^2h^2/\eta^2}{\gamma\sqrt{1+\gamma A+m^2h^2/\eta^2}}\Big)+ \Upsilon^3/3+ \Upsilon(X-A)\\+S^3/3+S(\frac{a}{\gamma}-A)
\nonumber
=\frac 43 NA^{3/2}(1-\frac 34 B'(\eta\lambda_{\gamma}A^{3/2})).
\end{gather}
Introducing the term $2N\sqrt{A}(1-\frac 34 B'(\eta\lambda_{\gamma}A^{3/2}))$ from \eqref{critA} in the second line of \eqref{criteta} provides a relation between $Y$ and $T$ that doesn't involve $N$ nor $B'$ as follows:
\begin{multline}\label{critYT}
Y+\frac{T}{\sqrt{1+\gamma A+m^2h^2/\eta^2}+\sqrt{1+\gamma}}\Big((A-1)-\frac{m^2h^2/\eta^2}{\gamma\sqrt{1+\gamma A+m^2h^2/\eta^2}}\Big)\\
+ \Upsilon^3/3+ \Upsilon(X-A)+S^3/3+S(\frac{a}{\gamma}-A)
=\frac 23 A\Big(\frac{T}{2\sqrt{1+\gamma A+m^2h^2/\eta^2}}-( \Upsilon+S)\Big).
\end{multline}
Let $t>h$ and let $N_j\in \mathcal{N}^m_{1,\gamma}(t,x,y)$, with $j\in \{1,2\}$ be any two elements of $\mathcal{N}^m_{1,\gamma}(t,x,y)$. Then there exists $(t_j,x_j,y_j)$ such $N_j\in\mathcal{N}^m_{\gamma}(t_j,x_j,y_j)$ ; writing $t_j=\sqrt{\gamma}T_j$, $x_j=\gamma X_j$, $y_j+t_j\sqrt{1+\gamma}=\gamma^{3/2}Y_j$ and rescaling $(t,x,y)$ as in \eqref{eq:83ff}, the condition below holds true
\[
|T_j-T|\leq 1, \quad |X_j-X|\leq 1, \quad |Y_j-Y|\leq 1.
\]
We prove that $|N_1-N_2|$ is bounded by $O(1)+m^2|T|h^2/\gamma+|T|/\lambda^2_{\gamma}$, which will achieve the first part of Proposition \ref{propcardN}. Since $N_j\in\mathcal{N}^m(t_j,x_j,y_j)$, it means that there exists $ \Upsilon_j,A_j,\eta_j,S_j$ such that \eqref{critSig}, \eqref{critA}, \eqref{criteta} holds with $T,X,Y, \Upsilon,S,A,\eta$ replaced by $T_j,X_j,Y_j, \Upsilon_j,S_j,A_j,\eta_j$, respectively. We re-write \eqref{critA} as follows
\begin{equation}\label{critAj}
2N_j\sqrt{A_j}(1-\frac 34 B'(\eta_j\lambda_{\gamma}A^{3/2}_j))=\frac{T_j}{2\sqrt{1+\gamma A_j+m^2h^2/\eta^2_j}}-(\Sigma_j+S_j).
\end{equation}
Multiplying \eqref{critAj} by $\sqrt{A_{j'}}$, for $j,j'\in\{1,2\}$, $j'\neq j$, taking the difference and dividing by $\sqrt{A_1A_2}$ yields
\begin{multline}\label{difN1N2}
  2(N_1-N_2)=\frac 32 \Big(N_1B'(\eta_1\lambda_{\gamma}A_1^{3/2})-N_2B'(\eta_2\lambda_{\gamma}A_2^{3/2})\Big)
  -\frac{\Sigma_1+S_1}{\sqrt{A_1}}+\frac{\Sigma_2+S_2}{\sqrt{A_2}}\\
{}+\frac{T_1}{2\sqrt{A_1}\sqrt{1+\gamma A_1+m^2h^2/\eta^2_1}}-\frac{T_2}{2\sqrt{A_2}\sqrt{1+\gamma A_2+m^2h^2/\eta^2_2}}.
\end{multline}
Using that $\Sigma_j,S_j\lesssim A_j$, $A_j\simeq 1$, it follows that $\frac{\Sigma_j+S_j}{\sqrt{A_j}}=O(1)$, for $j\in \{1,2\}$. The first term, involving $B'$, in the right hand side of \eqref{difN1N2} behaves like $(N_1+N_2)/\lambda_{\gamma}^2$, which follows using $B'(\eta\lambda A^{3/2})\simeq -\frac{b_1}{\eta^2\lambda^2A^3}$ and $\eta,A\simeq 1$. Notice that we cannot take any advantage of the fact that we estimate a difference of two terms, since each $N_jB'(\eta_j\lambda_{\gamma}A_j^{3/2})$ corresponds to some $\eta_j,A_j$ (close to $1$) and the difference $\frac{1}{\eta_1A_1^{3/2}}-\frac{1}{\eta_2A_2^{3/2}}$ is bounded by a constant that has no reason to be small (the difference between $A_j$ turns out to be $O(1/T)$, but we don't have any information about the difference between $\eta_j$ which is simply bounded by a small constant on the support of $\psi$). Therefore the bound $(N_1+N_2)/\lambda_{\gamma}^2$ for the terms involving $B'$ in \eqref{difN1N2} is sharp. Since $N_j\simeq T_j$, and $|T_j-T|\leq 1$, it follows that this contribution is $\simeq |T|/\lambda^2_{\gamma}$. 
We are reduced to prove that the difference of the last two terms in the second line of \eqref{difN1N2} is $O(1)+m^2|T|h^2/\gamma$. Write 
\[
\frac{T_j}{2\sqrt{A_j}\sqrt{1+\gamma A_j+m^2h^2/\eta^2_j}}=\frac{T_j}{2\sqrt{A_j}\sqrt{1+\gamma A_j}}\Big(1-\frac{m^2h^2}{2\eta_j^2\sqrt{1+\gamma A_j}}+O(m^4h^4)\Big).
\]
The difference for $j=1,2$ of the contributions involving $\eta_j^{-2}$ cannot be estimated better than by $m^2h^2|T|$.  We now proceed, as in \cite{ILP3} in the case $m=0$, with the difference of the main terms :
\begin{multline}\label{estimdifT1T2}
\Big|\frac{T_1}{\sqrt{A_1}\sqrt{1+\gamma A_1}}-\frac{T_2}{\sqrt{A_2}\sqrt{1+\gamma A_2}}\Big|\leq \frac{|T_1-T_2|}{\sqrt{A_1}\sqrt{1+\gamma A_1}} %
+\frac{T_2}{\sqrt{A_{1}A_{2}}}\Big|\frac{\sqrt{A_2}}{\sqrt{1+\gamma A_1}}-\frac{\sqrt{A_1}}{\sqrt{1+\gamma A_2}}\Big|\\
\leq \frac{2}{\sqrt{A_1}\sqrt{1+\gamma A_1}}+\frac{T_2|A_2-A_1|(1+\gamma(A_1+A_2))}{\sqrt{A_1A_2(1+\gamma A_1)(1+\gamma A_2)}(\sqrt{A_1(1+\gamma A_1)}+\sqrt{A_2(1+\gamma A_2)})}\\
\leq C(1+T_2|A_2-A_1|),
\end{multline}
where $C>0$ is some absolute constant. We have only used $|T_2-T_1|\leq 2$ and $A_j\simeq 1$. Notice that for $T$ bounded we can conclude since $|T_2-T|\leq 1$. We are therefore reduced to bound $T_2|A_2-A_1|$ when $T_2$ is sufficiently large. For that, we need to take into account the $Y$ variable. We use \eqref{critYT} with $T,X,Y,\Sigma,S,A,\eta$ replaced by $T_j,X_j,Y_j,\Sigma_j,S_j,A_j,\eta_j$, $j\in\{1,2\}$ to eliminate the terms containing $N$ and $B'$ as follows:
\begin{multline}
Y_j+\Sigma^3_j/3+\Sigma_j(X_j-A_j)+S^3_j/3+S_j(\frac{a}{\gamma}-A_j)+\frac 23 A_j(\Sigma_j+S_j)\\
=T_j\Big(\frac{A_j}{3\sqrt{1+\gamma A_j}}-\frac{(A_j-1)}{\sqrt{1+\gamma A_j}+\sqrt{1+\gamma}}+O(m^2h^2/\gamma)\Big).
\end{multline}
If $|T|$ is sufficiently large then so is $|T_j|$, and we divide the last equation by $T_j$ in order to estimate the difference $A_1-A_2$ in terms of $Y_1/T_1-Y_2/T_2$ as follows
\begin{equation}\label{critYsurT}
\frac{Y_j}{T_j}+O(\frac{A_j^{3/2}}{T_j})=F_{\gamma}(A_j), \quad F_{\gamma}(A)=\frac{A}{3\sqrt{1+\gamma A}}-\frac{(A-1)}{\sqrt{1+\gamma A}+\sqrt{1+\gamma}}+O(m^2h^2/\gamma).
\end{equation}
Taking the difference of \eqref{critYsurT} written for $j=1,2$ gives
\[
(\frac{Y_2}{T_2}-\frac{Y_1}{T_1})+O(\frac{1}{T_1})+O(\frac{1}{T_2})+O(m^2h^2/\gamma)=(A_2-A_1)\int_0^1\partial_A F_{\gamma}(A_1+o(A_2-A_1))do.
\]
Using that $\partial_A F_{\gamma}(1)=-\frac{(1+2\gamma)}{6(1+\gamma)^{3/2}}$ and that $A_j$ are close to $1$, $T_j$ are large, it follows that we can express $A_2-A_1$ as a smooth function of $(\frac{Y_2}{T_2}-\frac{Y_1}{T_1})$ and $O(\frac{1}{T_1})+O(\frac{1}{T_2})$, with coefficients depending on $A_j$ and $\gamma$. Write
\[
A_2-A_1=H_{\gamma}\Big(\frac{Y_2}{T_2}-\frac{Y_1}{T_1},\frac{1}{T_1}, \frac{1}{T_2}\Big)+O(m^2h^2/\gamma)\simeq -6(\frac{Y_2}{T_2}-\frac{Y_1}{T_1})+O(\frac{1}{T_1},\frac{1}{T_2})+O(m^2h^2/\gamma),
\]
and replacing the last expression in the last line of \eqref{estimdifT1T2} yields
\begin{align}
T_2|A_2-A_1| & \lesssim T_2\Big|\frac{Y_2}{T_2}-\frac{Y_1}{T_1}\Big|+O(T_2/T_1)+O(1)+|T|O(m^2h^2/\gamma)\\
 & \lesssim |Y_2-Y_1|+Y_1|1-T_2/T_1|+O(T_2/T_1)+O(1)+|T|m^2h^2/\gamma=O(1)+m^2|T|h^2/\gamma,\nonumber
\end{align}
where we have used $|Y_1-Y_2|\leq 2$, $|T_1-T_2|\leq 2$ and that $Y_1/T_1$ is bounded (which can easily be seen from \eqref{critYsurT}). This ends the proof of \eqref{eq:113}.
The second statement of Proposition \ref{propcardN} follows exactly like in \cite[Proposition 1]{ILP3}.
\end{proof}

In the next three subsections we give an overview of dispersive estimates obtained for the wave equation (for $m=0$, see \cite{ILP3} for small $t$) and we generalize them to large time and for the Klein-Gordon flow (when $m=1$). 

\subsection{Tangential waves for $\gamma\simeq a\gtrsim h^{2(1-\varepsilon)/3}$}\label{sectang} This  corresponds to initial angles $|\xi|/|\theta|\lesssim \sqrt{a}$. We assume $\gamma/4\leq a\leq 4\gamma$. %
We rescale variables as follows:
\begin{equation}
  \label{eq:83ffa}
  x=aX\,, \alpha=a A\,, t=\sqrt a\sqrt{1+a} T\,, s=\sqrt a S\,, \sigma=\sqrt a \Upsilon\,, y+t\sqrt{1+a}=a^{3/2} Y\,.
\end{equation}
Define $\lambda=a^{3/2}/h >h^{-\varepsilon}$ to be our large parameter, then $V^m_{N,\gamma\simeq a}$ from \eqref{defVNgamma} becomes
\begin{equation}
  \label{eq:bis488terff}
  V^{m}_{N,\gamma\simeq a}(t,x,a,y)= \frac {a^{2}} {(2\pi h)^{3}} \int_{\R^{2}}\int_{\R^{2}} e^{i  \lambda \Psi^m_{N,a,a,h}}  \eta^{2} \psi_1({\eta})
 \psi_{2}(aA/\gamma) \, dS d\Upsilon  dA d\eta\,,
\end{equation}
where $\Psi^m_{N,a,a,h}$ is given by \eqref{PhiNagammadef} with $\gamma$ replaced by $a$ due to the change of variables.
As the support of $\psi_2$ is compact and as $\gamma/4\leq a\leq 4\gamma$, we find $A\in \frac{\gamma}{a}[\frac 34,2]\subset [3/16,8]$. Since the critical points satisfy $S^2=A-1$, $\Upsilon^2=A-X$, it follows that we can restrict to $|S|,|\Upsilon|< 3$ and $A\in [9/10,8]$ without changing the contribution of the integrals modulo $O(h^{\infty})$ (since for $A\leq 9/10$ the phase is non-stationary in $S$); we insert suitable cut-offs, $\chi_2(S)\chi_2(\Upsilon)$ supported in $[-3,3]$ and $\psi_3(A)$ supported in $[9/10,8]$ and obtain $V^m_{N,\gamma\simeq a}(t,x,a,y)=W^m_{N,a}(T,X,Y)+O(h^{\infty})$, where we abuse notations with respect to $\gamma$, replaced by $a$ ($\psi_3$ actually includes a harmless factor $a/\gamma$ in its argument, which we are hiding since it will play no role) :
\begin{equation}
  \label{eq:bis488ter9999}
W^{m}_{N,a}(T,X,Y)= \frac {a^{2}} {(2\pi h)^{3}}  \int_{\R^{2}}\int_{\R^{2}} e^{i \lambda  \Psi^m_{N,a,a,h}}  \eta^{2} \psi_1({\eta})  \chi_{2}(S)\chi_{2}(\Upsilon) \psi_{3}(A) \, dS d\Upsilon  dA d\eta\,.
\end{equation}

\begin{prop}\label{propN<lambda}(see \cite[Prop.2]{ILP3} for $m=0$)
Let $m\in\{0,1\}$. Let $|N|\lesssim\lambda$ and let $W^m_{N,a}(T,X,Y)$ be defined in \eqref{eq:bis488ter9999}. Then the stationary phase theorem applies in $A$ and yields, modulo $O(h^{\infty})$ terms
\begin{equation}
  \label{eq:bis488ter999999machin}
    W^m_{N,a}(T,X,Y)= \frac {a^{2}} {h^{3}(N\lambda)^{\frac 1 2}}  \int_{\R}\int_{\R^{2}} e^{i \lambda\Psi^m_{N,a,a,h}(T,X,\Upsilon,S,A_c,\eta)}  \eta^{2} \psi_1({\eta})
   \chi_{3}(S,\Upsilon,a,1/N,h,\eta) \, dS d\Upsilon d\eta\,,
\end{equation}
where $\chi_{3}$ has compact support in $(S,\Upsilon)$ and harmless dependency on the parameters $a,h,1/N,\eta$. 
\end{prop}
\begin{rmq}\label{rmqN<lambda}
For $m=0$, Proposition \ref{propN<lambda} had been proved in \cite[Prop.2]{ILP3}. Let $m=1$ and
 $|N|\lesssim \lambda$ : %
 we state that the factor 
$e^{\frac ih t\eta(\sqrt{1+aA+ m^2h^2/\eta^2}-\sqrt{1+aA})}$
can be brought into the symbol, in which case the phase of $W^{m=1}_{N,a}(T,X,Y)$ can be taken to be $\Psi^{m=0}_{N,a,a,h}$. As $\sqrt{1+aA+ m^2h^2/\eta^2}-\sqrt{1+aA}=O(h^2)$ and $\frac{t}{\sqrt{a}}\simeq N\lesssim \lambda=\frac{a^{3/2}}{h}$ (using Lemma \ref{lemmeNtpetit}), we find $th\lesssim a^2$, which yields $\frac th\times O(h^2)=th\lesssim 1$. 
\end{rmq}
\begin{prop}\label{propN>lambda}(see \cite[Prop.3]{ILP3} for $m=0$)
Let $m\in\{0,1\}$, $N\gg \lambda$ and $W^m_{N,a}(T,X,Y)$ be defined in \eqref{eq:bis488ter9999}, then the stationary phase applies in both $A$ and $\eta$ and yields
\begin{equation}
  \label{eq:bis488ter999999machin>}
    W^m_{N,a}(T,X,Y)= \frac {a^{2}} {h^{3} N} \int_{\R^{2}} e^{i \lambda \Psi^m_{N,a,a,h}(T,X,Y,\Upsilon,S,A_c,\eta_c)}  \chi_{3}(S,\Upsilon,a,1/N,h) \, dS d\Upsilon +O(h^{\infty})\,,
\end{equation}
where $\chi_{3}$ has compact support in $(S,\Upsilon)$ and harmless dependency on the parameters $a,h,1/N$. 
\end{prop}
\begin{proof}
The only new situation is the case $N\gtrsim \lambda$ when $m=1$, when we cannot eliminate the terms depending on $m^2h^2/\eta^2$ from the phase. In order to prove Proposition \ref{propN>lambda} for $m=1$, we notice that the main contribution of the determinant of the Hessian matrix (with respect to $A$ and $\eta$) comes from $NB(\eta\lambda A^{3/2})$ and not from the terms involving $m^2h^2/\eta^2$.  
The second order derivatives of phase $\Psi^{m=1}_{N,a,h}$ of $W^{m=1}_{N,a}(T,X,Y)$, defined in \eqref{PhiNagammadef} are given by 
\begin{gather}\nonumber
\partial^2_{\eta,\eta}(\Psi^{m=1}_{N,a,a,h})=N\lambda A^3B''(\eta\lambda A^{3/2})+O(m^2Th^2/a)\simeq \frac{N}{\lambda^2}+O(\frac{Th^2}{a})\simeq \frac{N}{\lambda^2},\\
\nonumber
\partial^2_{A,A}(\Psi^{m=1}_{N,a,a,h})=\simeq -\eta\frac{N}{A^{1/2}},\quad \partial^2_{\eta,A}(\Psi^{m=1}_{N,a,a,h})=\eta^{-1}\partial_A\Psi^{m=1}_{N,a,a,h}+\frac 32\eta \lambda NA^{2}B''(\eta\lambda A^{3/2}))+O(m^2Th^2).
\end{gather}
Since $B''(\eta\lambda A^{3/2})\simeq O(\lambda^{-3})$, then $\partial_{\eta,A}\Psi^{m=1}_{N,a,a,h}\sim N/\lambda^{2}$. At the critical points, the Hessian of $\Psi^{m=1}_{N,a,a,h}$ satisfies
\[
\text{det Hess }\Psi^{m=1}_{N,a,a,h}|_{\partial_{A}\Psi^{m=1}_{N,a,a,h}=\partial_{\eta}\Psi^{m=1}_{N,a,a,h}=0}\simeq \frac{N^2}{\lambda^2},\quad N\gtrsim \lambda.
\]
\end{proof}
We are left with the integration over $(S,\Upsilon)$. In \cite[Prop.4,5,6 $\&$ Corollary 1]{ILP3}, sharp bounds for $W^{m=0}_{N,a}$ had been obtained: we recall them as they will be crucial in order to prove Theorem \ref{thmDKG} for all $t$. As noticed in Remark \ref{rmqN<lambda}, when $|N|\lesssim \lambda$ the factor $e^{\frac ih a^{3/2}(\Psi^{m=1}_{N,a,a,h}-\Psi^{m=0}_{N,a,a,h})}$ may be brought into the symbol, hence the critical value of the phase $\Psi^m_{N,a,a,h}$ in Proposition \ref{propN<lambda} for $m=1$ is the same as the one for $m=0$. Therefore, the following results hold as in \cite[Prop.4,5,6,\& 7]{ILP3}:
\begin{prop}\label{propfreeWham}
Let $m\in\{0,1\}$. For $T\leq \frac 52$ we have
\[
\sum_{N}|W^m_{N,a}(T,X,Y)|=|W^m_{0,a}(T,X,Y)|+\sum_{\pm 1}|W^m_{\pm 1,a}(T,X,Y)|+O(h^{\infty})\lesssim \frac{1}{h^2}\inf\Big(1,\Big(\frac ht\Big)^{1/2}\Big).
\] 
\end{prop}
\begin{prop}
\label{propdispNpetitpres}
Let $m\in \{0,1\}$. For $1\leq N<\lambda^{1/3}$ and $|T-4N|\lesssim 1/N$, we have
  \begin{equation}
    \label{eq:2hh}
       \left| W^m_{N,a}(T,X,Y)\right| \lesssim  \frac 1 {h^{2}}\frac{h^{1/3}}{((N/\lambda^{1/3})^{1/4}+|N(T-4N)|^{1/6})}\,.
  \end{equation}
\end{prop}
\begin{proof}
\eqref{eq:2hh} is obtained from the following bounds obtained for $m\in\{0,1\}$ as in \cite[Prop.6]{ILP3} 
\begin{equation}\label{borneintWNsmallpres}
\Big|\int e^{i \lambda\Psi^m_{N,a,a,h}(T,X,\Upsilon,S,A_c,\eta)}  \eta^{2} \psi_1({\eta})
   \chi_{3}(S,\Upsilon,a,1/N,h,\eta) \, dS d\Upsilon d\eta\Big|\lesssim \frac{1}{N^2}\frac{(\lambda/N^3)^{-5/6}}{(\lambda/N^3)^{-1/12}+(N^2|\frac{T}{4N}-1|)^{1/6}}.
\end{equation}
\end{proof}

\begin{prop}
\label{propdispNpetitloin}
Let $m\in\{0,1\}$. For $1\leq N<\lambda^{1/3}$ and $|T-4N|\gtrsim 1/N$, we have
  \begin{equation}
    \label{eq:2ff}
       \left| W^m_{N,a}(T,X,Y)\right| \lesssim  \frac 1 {h^{2}}\frac{h^{1/3}}{(1+|N(T-4N)|^{1/2})}\,.
  \end{equation}
\end{prop}
\begin{proof}
\eqref{eq:2ff} is obtained from the following bounds obtained for $m\in\{0,1\}$ as in \cite[Prop.6]{ILP3} 
\begin{equation}\label{borneintWNsmall}
\Big|\int e^{i \lambda\Psi^m_{N,a,a,h}(T,X,\Upsilon,S,A_c,\eta)}  \eta^{2} \psi_1({\eta})
   \chi_{3}(S,\Upsilon,a,1/N,h,\eta) \, dS d\Upsilon d\eta\Big|\lesssim \frac{1}{N^2}\Big(\frac{\lambda}{N^3}\Big)^{-5/6}\frac{1}{1+(N^2|\frac{T}{4N}-1|)^{1/2}}.
\end{equation}
\end{proof}

\begin{prop}
\label{propdispNgrand}
Let $m\in \{0,1\}$. Let $\lambda^{1/3}\lesssim N$, then the following estimates hold true
\begin{enumerate}
\item When $\lambda^{1/3}\lesssim N\lesssim \lambda$,
  \begin{equation}
    \label{eq:1ff}
       \left| W^m_{N,a}(T,X,Y)\right|\lesssim \frac {1} { h^{2}}  \frac {h^{1/3}} {((N/\lambda^{1/3})^{1/2} +\lambda^{1/6}|T-4N|^{1/2})}\,.
  \end{equation}
    \item When $\lambda\lesssim N$, 
    \begin{equation}
    \label{eq:1ff>}
       \left| W^{m}_{N,a}(T,X,Y)\right|\lesssim \frac {1} { h^{2}}  \frac {h^{1/3}\sqrt{\lambda/N}} {(N/\lambda^{1/3})^{1/2}}\,.
  \end{equation}
  \end{enumerate}
\end{prop}
\begin{proof}
For $\lambda^{1/3}\lesssim N\lesssim \lambda$, \eqref{eq:1ff} follows by showing that the integral in \eqref{eq:bis488ter999999machin} satisfies
\begin{equation}\label{borneintW}
\Big|\int e^{i \lambda\Psi^m_{N,a,a,h}(T,X,\Upsilon,S,A_c,\eta)}  \eta^{2} \psi_1({\eta})
   \chi_{3}(S,\Upsilon,a,1/N,h,\eta) \, dS d\Upsilon d\eta\Big|\lesssim \frac{\lambda^{-2/3}}{1+\lambda^{1/3}|\frac{T}{4N}-1|^{1/2}},
\end{equation}
which, in turn, follows for both $m\in\{0,1\}$ as in the proof of \cite[Prop.7]{ILP3}.
From the discussion above, the only new situation is the case $|N|>\lambda^2$ when $m=1$, when we cannot eliminate the terms depending on $m^2h^2/\eta^2$ from the phase. In order to prove the last statement of Proposition \ref{propdispNgrand}, we notice that, after applying the stationary phase in both $A$ and $\eta$, we are essentially left with a product of two Airy functions (corresponding to the remaining integrals in $\Upsilon,S$) whose worst decay is $\lambda^{-2/3}$. Using \eqref{eq:bis488ter999999machin>}, we obtain
\[
|W^m_{N,a}(T,X,Y)|\lesssim \frac{a^2}{h^3N}\times \lambda^{-2/3}\simeq \frac{1}{h^2}\frac{1}{N}\frac{a^2}{h}\frac{h^{2/3}}{a}\simeq \frac{1}{h^2}\frac{1}{N}\frac{a}{h^{1/3}}=\frac{1}{h^2}\frac{1}{N}h^{1/3}\lambda^{2/3}= \frac {1} { h^{2}}  \frac {h^{1/3}\sqrt{\lambda/N}} {(N/\lambda^{1/3})^{1/2}}.
\]
\end{proof}
It should be clear from the previous estimates that there are different sub-regimes when studying decay, especially when $t$ is large. 
However, we need to stress that the proofs of Propositions \ref{propdispNpetitpres}, \ref{propdispNpetitloin} and \ref{propdispNgrand} in \cite{ILP3} only use the variable with $T=\frac{t}{\sqrt{a}}$ (which is compared to different powers of $\lambda$) and not $t$, which can therefore be taken as large as needed. 

\subsection{"Almost transverse" waves $\gamma>\max\{4a,h^{2/3(1-\varepsilon)}\}$}\label{sectransv}
Let $\max\{4a,h^{2/3(1-\epsilon)}\}<\gamma\lesssim 1$ and %
\begin{equation}
  \label{eq:83}
  x=\gamma X\,,\quad \alpha=\gamma A\,,\quad t=\sqrt \gamma T\,,\quad s=\sqrt \gamma S\,,\quad \sigma=\sqrt \gamma \Upsilon\,,\quad y+t\sqrt{1+\gamma}=\gamma^{3/2} Y\,.
\end{equation}
Let $\lambda_{\gamma}=\gamma^{3/2}/h$ be the large parameter and $\Psi^m_{N,a,\gamma,h}$ be as in \eqref{PhiNagammadef}. As in the case $\gamma \simeq a$, we find $V^m_{h,\gamma}(t,x,a,y)=W^m_{N,\gamma}(T,X,Y)+O(h^{\infty})$, where
\begin{equation}
  \label{eq:bis488ter999999bis}
    W^m_{N,\gamma}(T,X,Y): = \frac {\gamma^{2}} {(2\pi h)^{3}}  \int_{\R^{2}}\int_{\R^{2}} e^{i \lambda_{\gamma} \Psi^m_{N,a,\gamma,h}}  \eta^{2} \psi_1({\eta})
 \chi_{2}(S)\chi_{2}(\Upsilon) \psi_{2}(A) \, dS d\Upsilon  dA d\eta\,.
\end{equation}
The following results hold for both $m\in\{0,1\}$ as in \cite[Prop.8,9,10 \& 11]{ILP3}:
\begin{prop}\label{proptranstsmall}
Let $m\in\{0,1\}$. For $0<T\leq \frac 52$, $t=\sqrt{\gamma}T$, we have
\[
\sum_{N}|W^m_{N,\gamma}(T,X,Y)|=|W^m_{0,\gamma}(T,X,Y)|+\sum_{\pm 1}|W^m_{\pm 1,\gamma}(T,X,Y)|+O(h^{\infty})\lesssim \frac{1}{h^2}\inf\Big(1,\Big(\frac ht\Big)^{1/2}\Big).
\] 
\end{prop}
\begin{prop}\label{propT<lambda2} (see \cite[Prop. 9, 10]{ILP3} for $m=0$)
  For $1\leq  T\simeq N\lesssim \lambda_{\gamma}^{2}$ we have 
  \[
 |W^m_{N,\gamma}(T,X,Y)|\lesssim  \frac{\gamma^2}{h^3} \frac{1}{ \sqrt{\lambda_{\gamma} N}}\times \lambda_{\gamma}^{-1/2-1/3}=\frac{h^{1/3}}{h^{2}}\frac{1}{\sqrt{N}}\,.
\]
As a consequence, for $\sqrt{\gamma}\leq t \lesssim \sqrt{\gamma}\lambda_{\gamma}^2$ %
we obtain
  \[
|G^{m}_{h,\gamma}(t,x,a,y)|\lesssim \sum_{N\in\mathcal{N}^m_{1,\gamma}(t,x,y)} |W^m_{N,\gamma}(T,X,Y)|\lesssim \frac{1}{h^2}(\frac ht)^{1/2}\lambda_{\gamma}^{1/6},
  \]
  since $\#\mathcal{N}^m_1(t,x,y)=O(1)$ for such values of $t$.
\end{prop}

\begin{prop}\label{propT>lambda2}(see \cite[Prop.9]{ILP3} for $m=0$)
For $T\gtrsim \lambda_{\gamma}^{2}$ we have
$|W^m_{N,\gamma}(T,X,Y)|\lesssim  \frac{\gamma^2}{h^{3}}\frac 1 { N} \frac{1}{\lambda^{5/6}_{\gamma}}$.%
\end{prop}
\begin{rmq}
When $t$ is very large (which was not the case in \cite{ILP3} !), Proposition \ref{propT>lambda2} is not helpful anymore as, due to the overlapping (and the fact that $\#\mathcal{N}^m_{1,\gamma}(t,\cdot)$ is proportional to $t$), summing-up over $N$ yields an important loss. In the next section we show that using the parametrix in the form of a spectral sum provides a sharp bound for $G^{m}_{h,\gamma}(t,x,a,y)$ in the "almost transverse" case.
\end{rmq}

\begin{rmq}
Notice that we obtained the exact same bounds for both $m=0$ and $m=1$. Using Remark \ref{rmqN<lambda}, it becomes clear that this holds for all $T\simeq |N|\lesssim \lambda_{\gamma}$ ; for $T\simeq N \gg \lambda_{\gamma}$, the stationary phase applies in both $A$ and $\eta$: for $m=1$ we act as in the proof of Proposition \ref{propN>lambda} and use the fact that the main contribution of the second derivative with respect to $\eta $ of $\Psi^{m=1}_{N,\gamma}$ comes from $NB(\eta\lambda_{\gamma}A^{3/2})$ and not from the terms involving $m^2h^2/\eta^2$ (and both have factors $N$ or $T$, with $N\sim T$). This is the precise statement related to the heuristics described in the introduction.
\end{rmq}

\subsection{The case $h^{2/3}\lesssim \gamma\lesssim h^{2(1-\varepsilon)/3}$ and the case $\#\mathcal{N}^m_{1,\gamma}(t,\cdot)$ unbounded}\label{secWGM}

When $\gamma\lesssim h^{2(1-\varepsilon)/3}$ writing a parametrix under the form over reflected waves \eqref{defVNgamma} is not useful anymore, since the parameter $\lambda_{\gamma}=\gamma^{3/2}/h$ is small and stationary phase arguments do not apply. In this case we can only work with the parametrix in the form \eqref{greenfctbis}. Notice that \eqref{greenfctbis} can also be used (with surprisingly good results) when the time is very large: in fact, since the number of waves that overlap is proportional with $t$, summing over $N$ when $t>1/h$ will eventually provide dispersive bounds that are (much) worst than those announced in Theorem \ref{thmDKG}. In particular, for $t/\sqrt{\gamma}\geq \lambda_{\gamma}^2$ we have to work with \eqref{greenfctbis} to prove Theorem \ref{thmDKG}. Recall that in \eqref{greenfctbis} (for $d=2$) the sum is taken for $k<1/h$ (as $\gamma< 1$ and $k\simeq \gamma^{3/2}/h$ on the support of $\psi_2$). We intend to apply the stationary phase for each integral and keep the Airy factors as part of  the symbol. Let 
\[
\phi_k=y\eta+t\eta\sqrt{1+\omega_kh^{2/3}\eta^{-2/3}+m^2h^2\eta^{-2}},
\]
then
\[
\partial_{\eta}\phi_k=y+t\frac{1+\frac 23\omega_kh^{2/3}\eta^{-2/3}}{\sqrt{1+\omega_kh^{2/3}\eta^{-2/3}+m^2h^2\eta^{-2}}},
\]
\begin{multline*}
\partial^2_{\eta,\eta}\phi_k=\frac{t}{\sqrt{1+\omega_kh^{2/3}\eta^{-2/3}+m^2h^2\eta^{-2}}^3}\Big(-\frac 19 \omega_kh^{2/3}\eta^{-5/3}(1+2\omega_k(h/\eta)^{2/3}-2m^2h^2/\eta^3)\\+m^2h^2/\eta^3\Big)\simeq -\frac t9\omega_kh^{2/3}\eta^{-5/3}.
\end{multline*}
Since $\eta\simeq 1$ on the support of $\psi$, for $t$ large enough the large parameter is $t\omega_k/h^{1/3}$. One needs to check that one has, for some $\nu>0$, $|\partial^j_{\eta}Ai((\eta/h)^{2/3}x-\omega_k)|\leq C_j\mu^{j(1/2-\nu)}$.
Since one has
\begin{equation}\label{derivelAirybounds}
\sup_{b>0}|b^lAi^{(l)}(b-\omega_k)|\lesssim \omega_k^{3l/2},\quad \forall l\geq 0,
\end{equation}
it will be enough to check that there exists $\nu>0$ such that, for every $k\lesssim 1/h$ and every $t$ sufficiently large, the following holds
\begin{equation}\label{tocheck}
\omega_k^{3/2}\lesssim (t\omega_k/h^{1/3})^{1/2-\nu}.
\end{equation}
As  $\omega_k\sim \gamma/h^{2/3}=\lambda^{2/3}_{\gamma}$ on the support of $\psi_2$,
proving \eqref{tocheck} for some $\nu>0$ is equivalent to showing that $\lambda_{\gamma}\lesssim (t\lambda_{\gamma}^{2/3}/h^{1/3})^{1/2-\nu}$ which in turn is equivalent to
$\gamma^3/h^2\lesssim (t\gamma/h)^{1-2\nu}$, which we further write
\begin{equation}\label{tocheck1}
\gamma (\gamma/h)^{1+2\nu}\lesssim t^{1-2\nu}.
\end{equation}
Let $t(h,\gamma,\nu):=\Big(\gamma (\gamma/h)^{1+2\nu}\Big)^{1/(1-2\nu)}$ for small $\nu>0$, then
the last inequality holds true for any $t\geq t(h,\gamma,\nu)$ and the stationary phase applies as long as $t\omega_k/h^{1/3}\gg 1$.
\begin{rmq}\label{rmqquandsumkmarche}
In particular, when the cardinal of $\mathcal{N}^m_{1,\gamma}(t,\cdot)$ is unbounded, we have
$t\geq \gamma^{1/2}\frac{\gamma^3}{h^2}$ and therefore the inequality \eqref{tocheck1} holds for $0<\nu\leq \frac 27\epsilon$. Hence, the stationary phase applies with the Airy factors as part of the symbol as soon as the number of overlapping waves is large.
\end{rmq}

\begin{lemma}\label{lemsob}(see \cite[Lemma 3.5]{Annals})
There exists $C$ such that for $L\geq 1$ the following holds true
\begin{equation}\label{estairy2}
\sup_{b\in \R}\Big (\sum_{1\leq k\leq L}\omega_k^{-1/2}Ai^2(b-\omega_{k})\Big)\leq CL^{1/3}, \quad
\sup_{b\in\R_+}\Big (\sum_{1\leq k\leq L}\omega_k^{-1/2}h^{2/3}Ai'^2(b-\omega_{k})\Big)\leq C_{0}h^{2/3}L.
\end{equation}
\end{lemma}

When \eqref{tocheck1} holds, we can apply the stationary phase in $\eta$ (with $\eta$ on the support $[\frac 12,\frac 32]$ of $\psi$). Let $\eta_c(\omega_k)$ denote the critical point that satisfies
$-6(\frac yt+1)=\omega_kh^{2/3}/\eta^{2/3}+O((\omega_kh^{2/3}/\eta^{2/3})^2)$.

\begin{itemize}
\item The "tangent" case $\gamma\sim a\gtrsim h^{2/3}$. As $h^{2/3}\omega_k\simeq \gamma\sim a$ on the support of the symbol, then $|\eta_c(\omega_k)/h)^{2/3}a- \omega_k|\lesssim 1$ and $Ai((\eta_c/h)^{2/3}x-\omega_k)$ stays close to $Ai(0)$.
For all $t\geq t(h,a,\nu)$, $\nu>0$ such that $t\omega_k/h^{1/3}\sim ta/h\gg 1$, the stationary phase with respect to $\eta$ yields, with $\lambda=a^{3/2}/h$ and $k\sim \lambda$ and $L\sim \lambda$ in \eqref{estairy2},
\begin{multline}\label{estimdispuhgamAiry}
|G^m_{h,a}(t,x,a,y)|\lesssim \frac{h^{1/3}}{h^2}\Big|\sum_{k\sim \lambda}\frac{\psi_2(h^{2/3}\omega_k/(\eta_c^{2/3}a))}{L'(\omega_k)} \sqrt{\frac{h^{1/3}}{t\omega_k}}Ai((\eta_c/h)^{2/3}x-\omega_k)Ai((\eta_c/h)^{2/3}a-\omega_k)\Big|\\
\lesssim \frac{1}{h^2}(\frac ht)^{1/2}\Big|\sum_{k\sim \lambda} k^{-2/3}\Big|\simeq \frac{1}{h^2}(\frac ht)^{1/2}\lambda^{1/3}.
\end{multline}

\item The "almost transverse" case $\gamma\geq \max\{4a,h^{2/3}\}$ : while for $x\sim\gamma$ the factor $Ai((\eta_c/h)^{2/3}x-\omega_k)$ may stay close to $Ai(0)$, for $\omega_k\sim (\eta_c/h)^{2/3}\gamma\geq 4 (\eta_c/h)^{2/3}a$, we obtain better decay from $|Ai((\eta_c/h)^{2/3}a-\omega_k)|\lesssim \frac{1}{1+\omega_k^{1/4}}$.
For $4\lambda\lesssim k\lesssim \frac 1h$ on the support of $\psi_2(\frac{h^{2/3}\omega_k}{\eta^{2/3}\gamma})$ and for $t\geq t(h,\gamma,\nu)$,
the Cauchy-Schwarz inequality together with Lemma \ref{lemsob} with $L\sim \lambda_{\gamma}$ yield %
\begin{multline}\label{estimuhgamtransv}
|G^m_{h,\gamma}(t,x,a,y)|\leq\frac{1}{h^2}(\frac ht)^{1/2}\Big|\sum_{k\sim\lambda_{\gamma}} \omega_k^{-1}Ai^2((\eta_c/h)^{2/3}x-\omega_k)\Big|^{1/2}\Big|\sum_{k\sim \lambda_{\gamma}} \omega_k^{-1}Ai^2((\eta_c/h)^{2/3}a-\omega_k)\Big|^{1/2}
\\\lesssim  \frac{1}{h^2}(\frac ht)^{1/2}\lambda_{\gamma}^{1/6}(\log (1/h))^{1/2},
\end{multline}
where we have used $\Big|\sum_{k\sim \lambda_{\gamma}} \omega_k^{-1}Ai^2((\eta_c/h)^{2/3}a-\omega_k)\Big|\leq \Big|\sum_{k\sim \lambda_{\gamma}} \frac{\omega_k^{-1}}{(1+\omega_k^{1/4})^2}\Big|\lesssim 
\log(1/h) $.
\end{itemize}

\section{Dispersive estimates for the wave flow in large time and the Klein Gordon flow. Proof of Theorem \ref{thmDKG} in the high frequency case $\sqrt{-\Delta_F}\simeq \frac 1h$ and $|\theta|\simeq \frac 1h$}\label{secdisptangHF}
\subsection{"Almost transverse" waves $\sum_{\max\{4a,h^{2/3}\}\leq \gamma\lesssim 1}G^m_{h,\gamma}$} Let $\max\{4a,h^{2/3}\}\leq \gamma \lesssim 1$.
\begin{prop}
For all $t>\sqrt{\gamma}$ and $m\in \{0,1\}$, we have $\Big|\sum_{\max\{4a, h^{2/3}\} \leq \gamma\lesssim 1}G^m_{h,\gamma}(t,\cdot)\Big|\lesssim  \frac{1}{h^2}(\frac ht)^{1/3}$, where the sum is taken over dyadic $\gamma$ as in \eqref{partunitpsi2}. 
\end{prop}
\begin{proof} 
If $t\leq  \sqrt{\gamma}$ we apply Proposition \ref{proptranstsmall} with $T=\frac{t}{\sqrt{\gamma}}\leq 1$ and then sum over $\gamma$ yields a bound $\frac{1}{h^2}\Big(\frac ht\Big)^{1/2}\log(1/h)$. Let $t\gtrsim \frac{1}{h^2}$, then $t>t(h,\gamma,\nu)$ for all $h^{2/3}\lesssim \gamma\lesssim 1$ and all $0\leq \nu<1/6$: applying the stationary phase in $\eta$ and using Lemma \ref{lemsob} with $L\sim \frac 1h$ yields, as in \eqref{estimuhgamtransv}
\begin{equation}\label{estimuhgamtransvsum}
|\sum_{4a\leq \gamma\lesssim 1} G^m_{h,\gamma}(t,\cdot)|%
\lesssim  \frac{1}{h^2}(\frac ht)^{1/2}\frac{1}{h^{1/6}}(\log (1/h))^{1/2}.
\end{equation}
As $(\frac ht)^{\frac{1}{18}}\leq h^{1/6}$ for all $t>\frac{1}{h^2}$, we find $|\sum_{4a\leq \gamma\lesssim 1} G^m_{h,\gamma}(t,\cdot)|\leq \frac{1}{h^2}(\frac ht)^{\frac 12-\frac{1}{18}}(\log (1/h))^{1/2}$ .
Notice that for such large values of $t$ these kind of bounds are much better then those obtained using Proposition \ref{propT>lambda2} and summing up over $N\in \mathcal{N}^m_{1,\gamma}(t,\cdot)$ (and then over $\cup_{\gamma\lesssim 1}$), because the cardinal of $\mathcal{N}^m_{1,\gamma}(t,\cdot)$ is proportional to $t/\sqrt{\gamma}$ and yields a bound independent of $t$ for the sum of $W^m_{N,\gamma}(t,\cdot)$.

We are left with $h^{1/3}\lesssim \sqrt{\gamma}\leq t\lesssim  \frac{1}{h^2}$ : let $t\sim \frac{\gamma_0^{7/2}}{h^2}$ for some $h^{2/3}<\gamma_0\lesssim 1$ (notice that $\gamma_0\sim h^{2/3}$ corresponds to $t\sim h^{1/3}$, while $\gamma_0\sim 1$ corresponds to $t\sim \frac{1}{h^2}$). Assume first $\gamma_0\geq h^{2/3(1-\varepsilon)}$ for some $\varepsilon>0$ : if, moreover, $\gamma_0>4a$, we  estimate separately the sums over $4a\leq \gamma\lesssim \gamma_0$ and $\gamma_0<\gamma\lesssim 1$. When $\gamma\lesssim \gamma_0$, then, using Remark \ref{rmqquandsumkmarche}, $t\sim \frac{\gamma_0^{7/2}}{h^2}\gtrsim \frac{\gamma^{7/2}}{h^2}=\frac{\gamma^2}{h}\lambda_{\gamma}>t(h,\gamma,\nu)$, for $0<\nu<\frac 27\epsilon$, and therefore we can use again  \eqref{estimuhgamtransv} with $\lambda_{\gamma}\leq \lambda_{\gamma_0}$. Then use $\lambda^{1/6}_{\gamma_0}\sim \gamma_0^{1/4}/h^{1/6}\lesssim(h/t)^{-1/14} h^{1/24}\ll (h/t)^{-1/6}$.

For dyadic $\gamma_0<\gamma\lesssim 1$ we have $t\sim \frac{\gamma_0^{7/2}}{h^2}<\frac{\gamma^{7/2}}{h^2}$ : in this case we use the form of $G^{m}_{h,\gamma}$ as a sum over reflected waves. According to Proposition \ref{propcardN} (for $m\in\{0,1\}$), $\# \mathcal{N}^m_{1,\gamma}(t,\cdot)=O(1)$, so only a bounded number of waves $W^m_{N,\gamma}(t,\cdot)$ overlap. Using Proposition \ref{propT<lambda2} for $h^{2/3(1-\epsilon)}\lesssim \gamma_0<\gamma\leq \min\{1,t^2\}$ yields
 \begin{equation}\label{sumtransvgam0>}
|\sum_{\gamma_0<\gamma\lesssim 1}G^{m}_{h,\gamma}(t,\cdot)|\leq \sum_{\gamma_0<\gamma\lesssim 1} \frac{1}{h^2}(\frac ht)^{1/2}\lambda_{\gamma}^{1/6} \leq \frac{1}{h^2}\Big(\frac ht\Big)^{1/2}\frac{1}{h^{1/6}}\min\{1,t^{1/2}\}.
\end{equation}
If $t\geq 1$, then $\frac{1}{h^{1/6}}\min\{1,t^{1/2}\}=\frac{1}{h^{1/6}}\leq (\frac th)^{-1/6}$, hence $(\frac ht)^{1/2}\frac{1}{h^{1/6}}\min\{1,t^{1/2}\}\leq (\frac ht)^{1/3}$ If $\sqrt{\gamma} \leq t\leq 1$, the last term in \eqref{sumtransvgam0>} is bounded by $h^{-2}h^{1/3}\leq h^{-2}(h/t)^{1/3}$.

Let now $\max\{4a,h^{2/3}\}<\gamma_0< h^{2/3(1-\epsilon)}$, for some small $0<\epsilon<1/7$, then $1\leq \lambda_{\gamma_0}\leq h^{-\epsilon}$. The sum over $h^{2/3(1-\epsilon)}\lesssim \gamma$ can be dealt with like in \eqref{sumtransvgam0>} : however, for $h^{1/3}\lesssim t\sim \frac{\gamma_0^{7/2}}{h^2}\leq h^{1/3(1-7\epsilon)}$ and $\gamma> h^{2/3(1-7\epsilon)}$ we notice that $\sum_{\sqrt{\gamma}N\sim O(t)}W^m_{N,\gamma}(t,\cdot)=W^m_{0,\gamma}(t,\cdot)+O(h^{\infty})$ (according to Lemma \ref{lemmeNtpetit}), and, as $|W^m_{0,\gamma}(t,\cdot)|\leq \frac{1}{h^2}(\frac ht)^{1/2}$, we have
\[
|\sum_{h^{2/3(1-7\epsilon)}<\gamma\lesssim 1}G^{m}_{h,\gamma}(t,\cdot)|\leq \sum_{\gamma_0<\gamma\lesssim 1} \frac{1}{h^2}(\frac ht)^{1/2}\log (1/h),
\]
while applying Proposition \ref{propT<lambda2} for $h^{2/3(1-\epsilon)}\lesssim \gamma\lesssim h^{2/3(1-7\epsilon)}$ yields 
\[
|\sum_{h^{2/3(1-\epsilon)}<\gamma\lesssim h^{2/3(1-7\epsilon)}}G^{m}_{h,\gamma}(t,\cdot)|\leq \sum_{\gamma_0<\gamma\lesssim 1} \frac{1}{h^2}(\frac ht)^{1/2}\lambda_{\gamma}^{1/6} \leq \frac{1}{h^2}\Big(\frac ht\Big)^{1/2}h^{-7\epsilon/6}.
\]
We are left with $\sum_{\gamma<h^{2/3(1-\epsilon)}}G^m_{h,\gamma}(t,\cdot)$ for $h^{1/3}\lesssim t\sim \frac{\gamma_0^{7/2}}{h^2}\leq h^{1/3(1-7\epsilon)}$.
We use \eqref{estairy2} (without the preliminary stationary phase in $\eta$) with $L\sim \frac{(h^{2/3(1-\epsilon)})^{3/2}}{h}=h^{-\epsilon}$, which yields, as $\frac ht\gtrsim h^{2/3(1+\epsilon/2)}$,
\[
\Big|\sum_{\gamma<h^{2/3(1-\epsilon)}}G^m_{h,\gamma}(t,\cdot)\Big|\lesssim \frac{h^{1/3}}{h^2}L^{1/3}=\frac{h^{1/3(1-\epsilon)}}{h^2}\leq \frac{1}{h^2}\Big(\frac ht)^{1/2}h^{-\epsilon/2}.
\]
Taking $\epsilon<\frac{2}{21}<1/7$ yields $\Big(\frac ht\Big)^{1/2}h^{-7\epsilon/6}<\Big(\frac ht\Big)^{1/3}$. The proof is achieved.
\end{proof}

\subsection{Tangential waves $G^m_{h,a}$}
We now focus on dispersive bounds for $G^m_{h,a}$ for $m\in\{0,1\}$.
\begin{prop} \label{propThmhh} 
There exists $C>0$ such that for every $h^{2/3}\lesssim a \lesssim 1 $, $h\in (0,1/2]$, $h<t\in \mathbb{R}_+$, the spectrally localized Green function $G^m_{h,a}(t,x,a,y)$ satisfies, for both $m\in\{0,1\}$,
\begin{equation}\label{dispomegamh}
|G^m_{h,a}(t,x,a,y)|\lesssim \frac{1}{h^{2}}\Big(a^{1/4}(\frac {h}{t})^{1/4}+(\frac ht)^{1/3}\Big), \quad \text{ for all } (x,y)\in \Omega_d.
\end{equation}
\end{prop}
\begin{proof}(of Proposition \ref{propThmhh})
Recall from Proposition \ref{propcardN}
that when $a>h^{2/3(1-\epsilon)}$ for some small $\epsilon>0$,
\begin{equation*}
  \left| \mathcal{N}^m_{1,a}(t,x,y)\right| = O(1)+m^2\frac{|t|h^2}{a^{3/2}}+\frac{|t|}{a^{1/2} \lambda^2}\,,\quad \lambda= a^{3/2}/h.
\end{equation*}
When $m=1$, the term $\frac{|t|h^2}{a^{3/2}}$ is always much smaller than $\frac{|t|}{a^{1/2} \lambda^2}$ and therefore it doesn't modify in any way the number of waves that overlap. Since all the estimates obtained in Sections \ref{sectang} and \ref{secWGM} are exactly the same for $G^m_{h,a}$ for both $m=0$ and $m=1$, we do not make the difference between these two cases. In the following we describe all the situations according to the size of $t$ and $a$. Notice first that if $a\ll h^{2/3}$ there is no contribution coming from tangent directions, so we consider only values for the initial distance from the boundary such that $a\gtrsim h^{2/3}$. Let $t>0$.
\vskip2mm

Let first $h^{2/3(1-\varepsilon)}\lesssim a\lesssim 1$ for some small $\varepsilon>0$ and consider $t\geq \frac{a^{7/2}}{h^2}$ : in this case $\#\mathcal{N}^m_{1,a}(t,\cdot)$ is large and the number of waves $W^m_{N,a}(t,\cdot)$ that cross each other behaves like $\frac{t}{\sqrt{a}\lambda^2}$. As $t\geq t(h,a,\nu)$ for $0<\nu<\frac 27\epsilon$, we use \eqref{estimdispuhgamAiry} to bound
\[
|G^m_{h,a}(t,\cdot)|\lesssim \frac{1}{h^2}(\frac ht)^{1/2}\lambda^{1/3}\ll \frac{1}{h^2}(\frac ht)^{1/3},
\]
where the last equality holds for all $t\gg h\lambda^2$; as $t\geq \sqrt{a}\lambda^2$ and $\sqrt{a}\gtrsim h^{1/3}\gg h$, this is obviously true.
Consider now $a>h^{2/3(1-\epsilon)}$ and $t\leq \frac{a^{7/2}}{h^2}$ when only a finite number of $W^m_{N,a}(t,\cdot)$ can meet : sharp dispersive bounds are then provided by Propositions \ref{propdispNpetitpres}, \ref{propdispNpetitloin} and \ref{propdispNgrand}. If $t>\frac{a}{h^{1/3}}=\sqrt{a}\lambda^{1/3}$ then Proposition \ref{propdispNgrand} applies and we obtain, as only a bounded number of $N$ are involved in the sum
\begin{itemize}
\item If $t\in [\frac{a}{h^{1/3}},\frac{a^2}{h})$, which corresponds to $\frac{t}{\sqrt{a}}\in (\lambda^{1/3},\lambda)$, then, using \eqref{eq:1ff}, we find
\begin{equation}\label{estimGhapire1}
|G_{h,a}^m(t,\cdot)|=|\sum_{ \frac{t}{\sqrt{a}}\sim N\in\mathcal{N}^m_{1,a}(t,\cdot)}W^m_{N,a}(\frac{t}{\sqrt{a}},\cdot) |\lesssim \frac{1}{h^2}\frac{h^{1/6}\sqrt{a}}{\sqrt{t}}\lesssim \frac{1}{h^2}(\frac ht)^{1/4}a^{1/4},
\end{equation}
where the last inequality holds for $t>\frac{a}{h^{1/3}}$. As $(\frac ht)^{1/4}a^{1/4}\leq (\frac ht)^{1/3}$ only if $t\leq \frac{h}{a^{3}}$, it follows that the (RHS) of \eqref{estimGhapire1} can be bounded by $\frac{1}{h^2}(\frac ht)^{1/3}$ only if $\frac{a}{h^{1/3}}\leq \frac{h}{a^3}$, hence for $a\leq h^{1/3}$. In particular, when $h^{1/3}<a \lesssim 1$, then $1\leq t\lesssim 1/h$ and the previous estimate is sharp.

\item If $t\in [\frac{a^2}{h}, \frac{a^{7/2}}{h^2})$, which corresponds to $\frac{t}{\sqrt{a}}\in (\lambda,\lambda^2)$, then, using \eqref{eq:1ff>}, we find
\[
|G_{h,a}^m(t,\cdot)|=|\sum_{ \frac{t}{\sqrt{a}}\sim N\in\mathcal{N}^m_{1,a}(t,\cdot)}W^m_{N,a}(t,\cdot) |\lesssim \frac{1}{h^2}\frac{h^{1/3}\lambda^{1/2+1/6}}{t/\sqrt{a}}=\frac{1}{h^2}\frac{a^{3/2}}{th^{1/3}}\leq  \frac{1}{h^2}(\frac ht)^{1/3},
\]
where the last inequality holds for all $t>\frac{a^{9/4}}{h}$ and therefore for all $t\geq \frac{a^2}{h}\gtrsim \frac{a^{9/4}}{h}$. 
\end{itemize}
Let now $\sqrt{a}\lesssim t<\frac{a}{h^{1/3}}$, then $1\leq \frac{t}{\sqrt{a}}\leq\lambda^{1/3}$ and Proposition \ref{propdispNpetitpres} yields a sharp bound
\begin{equation}\label{estimGhapire}
|G^m_{h,a}(t,\cdot)|\leq \frac{1}{h^2}(\frac ht)^{1/4}a^{1/4},
\end{equation}
which is reached at $t=t_n=4n\sqrt{a}\sqrt{1+a}$ for $1\lesssim n\leq \frac{\sqrt{a}}{h^{1/3}}$. As $(\frac ht)^{1/4}a^{1/4}<(\frac ht)^{1/3}$ only if $t\leq \frac{h}{a^{3}}$, it follows that when $a\leq h^{1/3}$ we have $t< \frac{a}{h^{1/3}}\leq \frac{h}{a^{3}}$ and $|G^m_{h,a}(t,\cdot)|\leq \frac{1}{h^2}(\frac ht)^{1/3}$ everywhere on $[\sqrt{a},\frac{a}{h^{1/3}}]$, while for $a>h^{1/3}$, \eqref{estimGhapire} is sharp on $(\frac{h}{a^{1/3}},\frac{a}{h^{1/3}}]$ and becomes $|G^m_{h,a}(t,\cdot)|\leq \frac{1}{h^2}(\frac ht)^{1/3}$ on $[\sqrt{a},\frac{h}{a^{3}}]$. Remark that in the last case $a>h^{1/3}$ we must have $t\lesssim 1$, which is the situation considered in \cite{Annals} (and \cite{ILP3}). Let $h<t\leq \sqrt{a}$, then $|G^m_{h,a}(t,\cdot)|\leq \frac{1}{h^2}(\frac ht)^{1/2}$ using Proposition \ref{propfreeWham}.

We are left with the case $h^{2/3}\lesssim a<h^{2/3(1-\epsilon)}$ for a small $\epsilon>0$ and with $t\leq \frac{a^{7/2}}{h^2}$. One can easily check that for $t>h^{1/3-2\epsilon}$, the condition \eqref{tocheck1}
 is satisfied for $0<\nu< \epsilon$. As $\mu:=t\omega_k/h^{1/3}\geq t/h^{1/3}\geq h^{-2\epsilon}$ is large, the stationary phase in $\eta$ applies and \eqref{estimdispuhgamAiry} holds, yielding
\[
|G^m_{h,a}(t,\cdot)|\lesssim \frac{1}{h^2}(\frac ht)^{1/2}\lambda^{1/3}\ll \frac{1}{h^2}(\frac ht)^{1/2}h^{-\epsilon/3}\lesssim  \frac{1}{h^2}(\frac ht)^{1/3},
\]
where the sum in $G^m_{h,a}$ is restricted to $k\lesssim \lambda\leq h^{-\epsilon}$ and where the last inequality holds for $t\geq h^{1/3-2\epsilon}\gg h^{1-2\epsilon}$. Let now $t\leq h^{1/3-2\epsilon}$.
Using Lemma \ref{lemsob} with $L\sim h^{-\epsilon}$ together with the Cauchy-Schwarz inequality yields 
\begin{equation}\label{g3-3}
| G^m_{h,a}(t,\cdot)|\lesssim h^{-2}h^{1/3}L^{1/3}=\frac{h^{1/3(1-\epsilon)}}{h^2}\,.
\end{equation}
If $t\lesssim h^{1/3-2\epsilon}$, then $h^{1/3(1-\epsilon)}\lesssim (\frac ht)^{1/2(1-4\epsilon)}$ ; if $t\lesssim h^{1/3(1+2\epsilon)}\ll a$, then $|G^m_{h,a}(t,\cdot)|\lesssim \frac{1}{h^2}(\frac ht)^{1/2}$.
\end{proof}

\section{Dispersive estimates for the wave flow in large time and the Klein Gordon flow. Proof of Theorem \ref{thmDKG} in the high frequency case\\ '$\sqrt{-\Delta_F}\simeq \frac 1h$" and $|\theta|\simeq \frac{1}{\tilde h}$, where $\tilde h>2h$}

In this section we consider again $d=2$ as the higher dimensional case can be dealt with exactly like in (the end of) Section \ref{secparamHF} and prove the following :
\begin{prop} \label{propThmhhtilde} 
Let $G^m_{(h,\tilde h)}(t,x,a,y)$ be as in \eqref{greenfctHFsmallertheta} and $m\in\{0,1\}$.
There exists a uniform constant $\tilde C>0$ such that for every $h\in (0,1/2]$, $2h\leq \tilde h $, $a\lesssim (\tilde h/h)^2$ and $t>0$, 
\begin{equation}\label{dispomegahtildeh}
| G^m_{(h,\tilde h)}(t,x,a,y)|\leq \tilde C
\left\{ \begin{array}{l} 
 \frac{1}{ h^{2}}(\frac {h}{t})^{1/4}(\frac{h}{\tilde h}), \quad \text{ if } h\leq t(h/\tilde h)^2< (\frac{\tilde h^2}{h^3})^{1/3} \\%
 \frac{1}{h^2}(\frac ht)^{1/3} h^{1/6} (\frac{h}{\tilde h})^{4/3},  \quad \text{ if } t(h/\tilde h)^2\geq (\frac{\tilde h^2}{h^3})^{1/3}.
  \end{array} \right.
\end{equation}
For $a>4 (\tilde h/h)^2$, $ G^m_{(h,\tilde h)}(t,x,a,y)\in O(h^{\infty})$. Moreover, for all $t>h$ we have
\[
|G^{\flat,m}_{h}(t,x,a,y)|=|\sum_{j\geq 1}G^{m}_{h,2^jh}(t,x,a,y)|\leq  \frac{1}{ h^{2}}(\frac {h}{t})^{1/4}.
\]
\end{prop}
In the remaining part of this section we prove Proposition \ref{propThmhhtilde}.
In \eqref{greenfctHFsmallertheta}, the symbols $\psi(\tilde \theta)$ and $\psi_1(h\sqrt{\lambda_k(\theta)})$ are supported for $\tilde h\theta\simeq 1$ and $h^2\lambda_k(\theta)\simeq 1$, which implies $\omega_k\simeq \frac{\tilde h^{4/3}}{h^2}$ and $k\simeq \frac{\tilde h^2}{h^3}$. Let $\theta=\eta/\tilde h$, with $\eta\simeq 1$ on the support of $\psi$. Remark that we must have $a\eta^{2/3}/\tilde h^{2/3}\leq\frac 32 \omega_k$ since otherwise, due to the behavior of the Airy function, the corresponding integral yields a $O(h^{\infty})$ contribution. The condition $\omega_k\simeq \frac{\tilde h^{4/3}}{h^2}$ (which holds on the support of $\psi_1$) implies $a\lesssim (\tilde h/h)^2$. When $a\leq \frac 14  (\tilde h/h)^2$ and $x\leq a$, both Airy factors $e_k(x,\eta/\tilde h)$ and $e_k(a,\eta/\tilde h)$ have only non degenerate critical points ; when $a\simeq (\tilde h/h)^2$, the Airy factor $e_k(a,\eta/\tilde h)$ may have degenerate critical points of order $2$ (and the same thing may happen to $e_k(x,\eta/\tilde h)$ for $x$ close to $a$). 

We let $a:=(\tilde h/h)^2 \tilde a $, then $\tilde a\lesssim 1$ and $a\theta^{2/3}= \tilde a \eta^{2/3}\tilde h^{4/3}/h^2$. 
In the two dimensional case, the phase of \eqref{greenfctHFsmallertheta} becomes
\begin{equation}\label{phasehhtilde}
\tilde \phi_{k,(h,\tilde h)}:=\frac{y\eta}{\tilde h}+\frac{t\eta}{h}\sqrt{\omega_k h^2/(\eta^{2/3}\tilde h^{4/3})+(h/\tilde h)^2+h^2m^2/\eta^2}
\end{equation}
and the Airy phase functions transform into 
\[
s^3/3+s\eta^{2/3}(\tilde h^{4/3}/h^2)(\tilde a-\omega_k h^2/(\eta^{2/3}\tilde h^{4/3}))+\sigma^3/3+\sigma\eta^{2/3}(\tilde h^{4/3}/h^2)(\frac xa \tilde a -\omega_k h^2/(\eta^{2/3}\tilde h^{4/3})).
\]
Using $1/2>h/\tilde h$ and that $\omega_k h^{2}/(\eta^{2/3}\tilde h^{4/3})\simeq 1$, it follows that the main term in the phase is $\frac{t\eta}{h}\sqrt{\omega_k h^2/(\eta^{2/3}\tilde h^{4/3})}$.
We show that we can reduce the analysis to the the previous case (when $\tilde h=h$).
Applying the Airy-Poisson formula \eqref{eq:AiryPoissonBis} and replacing the Airy factors by their integral formulas (see \eqref{eq:bis47}) allow to write $G^m_{(h,\tilde h)}$ under the following form
\begin{multline*}
G^{m}_{(h,\tilde h)}(t,x,a,y)=\frac{\tilde h^{1/3}}{(2\pi)^3\tilde h^2}\sum_{N\in\mathbb{Z}}\int_{\mathbb{R}}\int_{\mathbb{R}}\int_{\mathbb{R}^2}e^{\frac{i}{\tilde h}(y\eta+t\eta \frac{\tilde h}{h}\sqrt{\omega h^2/(\eta^{2/3}\tilde h^{4/3})+(h/\tilde h)^2+h^2m^2/\eta^2})-iNL(\omega)}\psi(\eta)\eta^{2/3}\\
\times \psi_1(\sqrt{h^2 \omega \eta^{4/3} /\tilde h^{4/3}+\eta^2(h/\tilde h)^2+h^2m^2})e^{i\eta(s^3/3+s(a\eta^{2/3}/\tilde h^{2/3}-\omega)+\sigma^3/3+\sigma(x\eta^{2/3}/\tilde h^{2/3}-\omega))}ds d\sigma d\omega d\eta,
\end{multline*}
where on the support of $\psi_1$ we have $\omega\simeq \frac{\tilde h^{4/3}}{h^2}$.
Set $T:=t(\frac{h}{\tilde h})^2$, $X:=x\frac{h^2}{\tilde h^2}$, $Y:=\frac{h^3}{\tilde h^3}(y+t)$ 
and rescale 
\[
\omega=(\tilde h/h)^{4/3}(\eta/h)^{2/3}\tilde A,\quad s=(\eta/h)^{1/3}(\tilde h/h)^{2/3}S,\quad \sigma=(\eta/h)^{1/3}(\tilde h/h)^{2/3}\Upsilon.
\]
To keep the same notations as in the previous section we replace the symbol $\psi_1$ by $\psi_2(\tilde A)$, with $\psi_2\in \mathcal{C}^{\infty}_0([\frac 14, 2])$ (this is possible without changing the contribution of the integrals modulo $O(h^{\infty})$ due to the support properties of $\psi_1$ and $\psi$); in the new variables we have
\begin{equation}\label{Gtildeform}
G^{m}_{(h,\tilde h)}(t,x,a,y)=\frac{1}{(2\pi)^3\tilde h^3}(\frac{\tilde h}{h})^4\sum_{N}\int_{\mathbb{R}}\int_{\mathbb{R}}\int_{\mathbb{R}^2}e^{i\frac{\tilde h^2}{h^3}\tilde \Psi^m_{N,\tilde a,(h,\tilde h)}(T,X,Y,\Upsilon,S,\tilde A,\eta)}\psi(\eta)\psi_2(\tilde A) d\Upsilon dS d\tilde A d\eta,\\
\end{equation}
where $\tilde A\simeq 1$, $\tilde a\lesssim 1$ and where we have defined (compare with \eqref{PhiNagammadef})
\begin{multline}\label{PhiNagammadeftilde}
\tilde \Psi^m_{N,\tilde a,(h,\tilde h)}(T,X,Y,\Upsilon,S,\tilde A,\eta):=
\eta\left(Y+ \Upsilon^3/3+ \Upsilon(X-\tilde A)+S^3/3+S(\tilde a-\tilde A)\right.\\
    {} +T(\sqrt{\tilde A+(h/\tilde h)^2+m^2h^2/\eta^2}-(h/\tilde h)
    {}-\left.\frac 43 N\tilde A^{3/2}\right)+\frac{N}{(\tilde h^2/h^3)}B(\eta(\tilde h^2/h^3)\tilde A^{3/2}).
\end{multline}
We define $\tilde \lambda:=\frac{\tilde h^2}{h^3}$, which satisfies $\tilde\lambda>\frac 1h$ and $G^m_{(h,\tilde h)}$ has now a  form similar to $G^m_{h,\gamma =1}$ (with $\lambda_{\gamma}$ replaced by $\tilde \lambda$). As in Section \ref{sectang}, using the equation satisfied by the critical points and the support of $\psi_2$, we can restrict ourselves to $|S|,\Upsilon|\leq 3$ and insert suitable cut-offs $\chi(S)\chi(\Upsilon)$ without changing the contribution of the integrals in \eqref{Gtildeform} modulo $O(\tilde\lambda^{-\infty})=O(h^{\infty})$. We can write $G^{m}_{(h,\tilde h)}(t,x,a,y)=\sum \tilde W^m_N(T,X,Y)$, where $\tilde W^m_N(T,X,Y)$ has the same form as $W^m_{N,a}(T,X,Y)$ in \eqref{eq:bis488ter9999} but where $\gamma$ is replaced by $1$, the factor $a^2$ is replaced by $(\frac{\tilde h}{h})^4$ (and is due to the change of variables in $\omega,s,\sigma$), $\lambda$ is replaced by $\tilde \lambda$ and $\Psi^m_{N,a,a,h}$ is replaced by $\tilde \Psi^m_{N,\tilde a,(h,\tilde h)}$ as follows
\begin{equation}\label{eq:bis488ter9999tilde}
 \tilde W^m_N(T,X,Y)=\frac{1}{(2\pi)^3\tilde h^3}(\frac{\tilde h}{h})^4\int_{\mathbb{R}}\int_{\mathbb{R}}\int_{\mathbb{R}^2}e^{i\tilde\lambda \tilde \Psi^m_{N,\tilde a,(h,\tilde h)}(T,X,Y,\Upsilon,S,\tilde A,\eta)}\psi(\eta)\psi_2(\tilde A)\chi(S)\chi(\Upsilon) d\Upsilon dS d\tilde A d\eta.
\end{equation}
We obtain the equivalent of Propositions \ref{propN<lambda} and \ref{propN>lambda} :
\begin{prop}\label{propN<lambdatilde }
Let $m\in\{0,1\}$. Let $|N|\lesssim\tilde\lambda$ and let $\tilde W^m_{N}(T,X,Y)$ be defined in \eqref{eq:bis488ter9999tilde}. Then the stationary phase theorem applies in $\tilde A$ and yields, modulo $O(\tilde\lambda^{-\infty})=O(h^{\infty})$ terms
\begin{equation}
  \label{eq:bis488ter999999machintilde}
    \tilde W^m_{N}(T,X,Y)= \frac {(\frac{\tilde h}{h})^4} {\tilde h^{3}(N\tilde\lambda)^{\frac 1 2}}  \int_{\R}\int_{\R^{2}} e^{i \lambda\Psi^m_{N,a,a,h}(T,X,\Upsilon,S,A_c,\eta)}  \eta^{2} \psi({\eta})
   \chi_{3}(S,\Upsilon,a,1/N,h,\eta) \, dS d\Upsilon d\eta\,,
\end{equation}
where $\chi_{3}$ has compact support in $(S,\Upsilon)$ and harmless dependency on the parameters $a,h,1/N,\eta$. 
\end{prop}
\begin{prop}\label{propN>lambdatilde}
Let $m\in\{0,1\}$, $|N|\gg \tilde \lambda$ and $\tilde W^m_{N}(T,X,Y)$ be defined in \eqref{eq:bis488ter9999tilde}, then the stationary phase applies in both $A$ and $\eta$ and yields modulo $O(\tilde\lambda^{-\infty})=O(h^{\infty})$ terms
\begin{equation}
  \label{eq:bis488ter999999machin>tilde}
    \tilde W^m_{N}(T,X,Y)= \frac {(\frac{\tilde h}{h})^4} {\tilde h^{3} N} \int_{\R^{2}} e^{i \lambda \Psi^m_{N,a,a,h}(T,X,Y,\Upsilon,S,A_c,\eta_c)}  \chi_{3}(S,\Upsilon,a,1/N,h) \, dS d\Upsilon \,,
\end{equation}
where $\chi_{3}$ has compact support in $(S,\Upsilon)$ and harmless dependency on the parameters $a,h,1/N$. 
\end{prop}
Using \eqref{borneintWNsmallpres}, \eqref{borneintWNsmall} and \eqref{borneintW} with $\lambda$ is replaced by $\tilde \lambda$ we obtain the equivalent of Propositions \ref{propdispNpetitpres}, \ref{propdispNpetitloin} and \eqref{propdispNgrand} where the only difference is that $h$ is replaced by $\tilde h$, $a$ by $\tilde a$, $\lambda$ by $\tilde \lambda$ and we have to distinguish three main situations according to whether $N< \tilde \lambda^{1/3}$, $\tilde\lambda^{1/3} \leq N\leq \tilde \lambda$ and $N>\tilde \lambda$. 

\begin{prop}
\label{propdispNpetitprestilde}
Let $m\in \{0,1\}$. For $1\leq |N|<\tilde \lambda^{1/3}$ and $|T-4N|\lesssim 1/N$, we have, when $\tilde a\sim 1$,
  \begin{equation}
    \label{eq:2hhtilde}
       \left| \tilde W^m_{N}(T,X,Y)\right| \lesssim  \frac{\tilde h^{1/3}}{\tilde h^2}\frac{1}{((N/\tilde \lambda^{1/3})^{1/4}+|N(T-4N)|^{1/6})}\,.
  \end{equation}
\end{prop}

\begin{prop}
\label{propdispNpetitlointilde}
Let $m\in\{0,1\}$. For $1\leq |N|<\tilde \lambda^{1/3}$ and $|T-4N|\gtrsim 1/N$, we have
  \begin{equation}
    \label{eq:2fftilde}
       \left| \tilde W^m_{N}(T,X,Y)\right| \lesssim  \frac{\tilde h^{1/3}}{\tilde h^2}\frac{1}{(1+|N(T-4N)|^{1/2})}\,.
  \end{equation}
\end{prop}

\begin{prop}
\label{propdispNgrandtilde}
Let $m\in \{0,1\}$. Let $\tilde \lambda^{1/3}\lesssim N$, then the following estimates hold true
\begin{enumerate}
\item When $\tilde\lambda^{1/3}\lesssim N\lesssim \tilde \lambda$,
     $  \left| \tilde W^m_{N}(T,X,Y)\right|\lesssim \frac {\tilde h^{1/3}}{\tilde h^2} \frac {1} {((N/\tilde \lambda^{1/3})^{1/2} +\tilde \lambda^{1/6}|T-4N|^{1/2})}\,$.
    \item When $\tilde \lambda\lesssim N(\lesssim T)$, 
      $ \left| \tilde W^{m}_{N}(T,X,Y)\right|\lesssim \frac {\tilde h^{1/3}}{\tilde h^{2}}\frac {\sqrt{\tilde \lambda/N}} {(N/\tilde \lambda^{1/3})^{1/2}}\,$.
  \end{enumerate}
Moreover, for $N\gg T$ we have $\tilde W^m_{N}(T,\cdot)=O(h^{\infty})$.
\end{prop}
When $N\gg T$, the last statement of Proposition \ref{propdispNgrandtilde} follows from Lemma \ref{lemmeNtpetit}.
The case $\tilde \lambda\lesssim N$ occurs when $N\simeq T=t(\frac{\tilde h}{h})^2> \frac 1h (\frac{\tilde h}{h})^2=\tilde\lambda$, hence for $t\gtrsim \frac{1}{h}$. Using Propositions \ref{propdispNpetitprestilde}, \ref{propdispNpetitlointilde} and \ref{propdispNgrandtilde} we can now follow exactly the same approach as in Section \ref{secdisptangHF} and obtain the corresponding estimates for each $\tilde h>2h$. We define a set $\mathcal{\tilde N}^m_{1}(T,X,Y)$ as follows 
\begin{gather*}
\mathcal{\tilde N}^m_1(T,X,Y)=\cup_{|(T',X',Y')-(T,X,Y)|\leq 1}\mathcal{\tilde N}^m(T,X,Y),\\
\mathcal{\tilde N}^m(T,X,Y)=\{N\in \mathbb{Z}, \exists (\Sigma,\Upsilon,\tilde A,\eta) \text{ such that } \nabla_{ (\Sigma,\Upsilon,\tilde A,\eta)}\tilde \Psi^m_{N,\tilde a,(h,\tilde h)}(T,X,Y,\Upsilon,S,\tilde A,\eta)=0\}.
\end{gather*}
\begin{prop}\label{propcardNtilde}
Let $t\in \mathbb{R}$, $t>h$, $\tilde h\geq 2h$ and $T=t(h/\tilde h)^2$. The following holds true:
\begin{itemize}
\item We control the cardinal of $\mathcal{\tilde N}^m_{1}(t,x,y)$,
\begin{equation}
  \label{eq:113tilde}
  \left| \mathcal{\tilde N}^m_{1}(T,X,Y)\right| \lesssim O(1)+T/\tilde\lambda^2+m^2h^2T\,,
\end{equation}
and this bound is optimal. 
\item The contribution of the sum over $N\notin \mathcal{\tilde N}^m_{1}(T,X,Y)$ in \eqref{uhgamN} is $O(\tilde \lambda^{-\infty})=O(h^{\infty})$.
\end{itemize}
\end{prop}
\begin{rmq}
Notice that when $m=1$ and $t$ is large enough, the main contribution in the right hand side of \eqref{eq:113tilde} comes this time from the last term, as $h/\tilde h\leq 1/2$. For $m=1$ and $t>\frac{1}{h^2}(\frac{\tilde h}{h})^2$, the cardinal of $\mathcal{\tilde N}^m_{1}(t,x,y)$ is $|t|h^2(\frac{\tilde h}{h})^2$; for $\frac{1}{h^2}(\frac{\tilde h}{h})^6<t\leq \frac{1}{h^2}(\frac{\tilde h}{h})^2$ the cardinal is $|t|h^2(\frac{\tilde h}{h})^2$.
In order to prove Proposition \ref{propcardNtilde} we use exactly the same approach as in the proof of Proposition \ref{propcardN} but with the phase function \eqref{PhiNagammadeftilde}. The parameter $\gamma$ is replaced by $1$, and the difference of any two difference points $N_{1,2}\in \mathcal{\tilde N}^m_1(t,x,y)$ is estimated (as in \eqref{difN1N2}) as follows, 
\[
|N_1-N_2|=O(1)+m^2h^2 T+\frac{T}{\tilde \lambda^2},\quad T=t(\frac{h}{\tilde h})^2,\quad  \tilde \lambda=\frac{\tilde h^2}{h^3}.
\]
\end{rmq}
As long as $T\leq \tilde \lambda^2$ (i.e. $t<\frac{1}{h^2}(\frac{\tilde h}{h})^6$) there is no overlap and we use Propositions \ref{propdispNpetitprestilde}, \ref{propdispNpetitlointilde} and \ref{propdispNgrandtilde}. For $N<\tilde\lambda^{1/3}$, the worst bound is given by \eqref{eq:2hhtilde} which reads as 
\begin{equation}\label{esttildeWN1}
|\tilde W^m_N(T,X,Y)|\lesssim \frac{\tilde h^{1/3}}{\tilde h^2} \frac{\tilde \lambda^{1/12}}{t^{1/4}( h/\tilde h)^{1/2}}=\frac{1}{h^2}(\frac{h}{t})^{1/4}(\frac{h}{\tilde h}),
\end{equation}
and this holds for $t$ such that $|t(\frac{h}{\tilde h})^2-4N|\lesssim \frac 1N$
and for $\tilde a\sim 1$, which is equivalent to $a\sim (\tilde h/h)^2$. 
When $T\sim N<\tilde\lambda$ we have $t(h/\tilde h)^2\leq \tilde h^{2/3}/h$, which implies 
\[
\frac{\tilde h^{1/3}}{\tilde h^2}\lesssim \frac{1}{h^2}(\frac{h}{t})^{1/4}(\frac{h}{\tilde h}).
\]
As, from \eqref{eq:2fftilde}, $\tilde W^m_N(T,X,Y)$ is always bound by $\tilde h^{1/3}/\tilde h^2$ for $N<\tilde\lambda^{1/3}$, it follows that the bounds \eqref{esttildeWN1} are sharp and hold for all $T\simeq N\lesssim \tilde \lambda^{1/3}$. 

Let now $\tilde\lambda^{1/3}\leq N\lesssim \tilde \lambda^2$ and $T\simeq N$. Using Proposition \ref{propdispNgrandtilde} $(1)$ yields, (as there is no overlap)
\[
|\tilde W^m_N(T,X,Y)|\lesssim \frac{\tilde h^{1/3}}{\tilde h^2} \frac{\tilde\lambda^{1/6}}{\sqrt{t} (h/\tilde h)}=\frac{\tilde h^{-1/3}}{h^{3/2}\sqrt{t}},
\]
and using that in this regime we gave $t(\frac{h}{\tilde h})^2\simeq N\geq \tilde\lambda^{1/3}=\frac{\tilde h^{2/3}}{h}$ we obtain $t\gtrsim \frac{\tilde h^{8/3}}{h^3}$. Introducing this in the last inequality yields
\begin{equation}\label{esttildeWN2}
|\tilde W^m_N(T,X,Y)|\lesssim \frac{\tilde h^{-1/3}}{h^{3/2}t^{1/4}}(\frac{h^3}{\tilde h^{8/3}})^{1/4}=\frac{1}{h\tilde h}(\frac{h}{t})^{1/4}=\frac{1}{h^2}(\frac{h}{t})^{1/4}(\frac{h}{\tilde h}),
\end{equation}
which is the same bound as in \eqref{esttildeWN1}. For $T> \tilde \lambda$, using Proposition \ref{propdispNgrandtilde} $(2)$ yields 
\begin{equation}\label{esttildeWN3}
|\tilde W^m_N(T,X,Y)|\lesssim \frac{\tilde h^{1/3}}{\tilde h^2}\frac{\tilde \lambda^{2/3}}{t(h/\tilde h)^2}\leq \frac{1}{h^2}(\frac{h}{t})^{1/3}(\frac{h}{\tilde h}).
\end{equation}
When $T>\tilde \lambda^2$ we use the form of the parametrix as a sum over eigenmodes $k$ and apply the stationary phase with respect to $\eta$ with the phase function given in \eqref{phasehhtilde} and the Airy factors in the symbols (as we did in Section \ref{secWGM}). For $\omega_k\simeq \frac{\tilde h^{4/3}}{h^2}$ and $\eta\simeq 1$ on the support of the symbol, we have, for $m\in\{0,1\}$ 
\[
|\partial^2_{\eta,\eta}\tilde \phi_{k,(h,\tilde h)}|\simeq \frac 29 \frac th \eta^{-4/3}(\omega_k\frac{h^2}{\tilde h^{4/3}})^{1/2}\Big(1+O((h/\tilde h)^2)+O(h^2)\Big).
\]
We may use again \eqref{derivelAirybounds}, where now $\omega_k^{3/2}\simeq \frac{\tilde h^2}{h^3}= \tilde \lambda$; the condition \eqref{tocheck}, necessary and sufficient in order to consider the Airy factors as part of the symbol, becomes
\[
\tilde \lambda\lesssim (t/h)^{1/2-\nu}
\]
for some $\nu>0$; as we have $T:=t(\frac{h}{\tilde h})^2>\tilde\lambda^2$, it follows that $(t/h)^{1/2}>\frac{1}{\sqrt{h}}\frac{\tilde h}{h}\tilde\lambda=\tilde\lambda^{3/2}$ and the last condition holds for $\nu=\frac 16$. As, after the change of coordinates $\theta=\eta/\tilde h$, $G^{m}_{(h,\tilde h)}(t,x,a,y)$ becomes
\begin{equation}
\sum_{k\geq 1}\frac{\tilde h^{1/3}}{\tilde h^2}\int_{\mathbb{R}}e^{i \tilde\phi_{k,(h,\tilde h)}}
 \psi(\eta)\eta^{2/3}\psi_{1}\Big(h\sqrt{\lambda_{k}(\eta/\tilde h)}\Big)
  \frac{1}{L'(\omega_k)} Ai (x\eta^{2/3}/\tilde h^{2/3} -\omega_k) Ai(a \eta^{2/3}/\tilde h^{2/3}-\omega_k)d\eta\,,
\end{equation}
it follows, applying the stationary phase in $\eta$ and using again Lemma \ref{lemsob}, that we have
\begin{align*}
|G^{m}_{(h,\tilde h)}(t,x,a,y)|&\lesssim \frac{\tilde h^{1/3}}{\tilde h^2}(\frac ht)^{1/2}\Big|\sum_{k}\frac{\psi_1\Big(\sqrt{\eta^{4/3}\omega_kh^2/\tilde h^{4/3}+\eta^2(h/\tilde h)^2+m^2h^2}\Big)}{L'(\omega_k)} \Big|\\
&\lesssim  \frac{\tilde h^{1/3}}{\tilde h^2}(\frac ht)^{1/2}\tilde \lambda^{2/3},
\end{align*}
where we have use the fact that on the support of the symbol we have $\eta\simeq 1$ and $k\simeq \omega_k^{3/2}\simeq \tilde\lambda$ and that $L'(\omega_k)\simeq \sqrt{2\omega_k}\simeq k^{1/3}$. As $t>\frac{1}{h^2}(\frac{\tilde h}{h})^6$, then $t^{-1/6}< h^{1/3}\frac{h}{\tilde h}$ and we obtain the second line in \eqref{dispomega} 
\[
\frac{\tilde h^{1/3}}{\tilde h^2}(\frac ht)^{1/2}\tilde \lambda^{2/3}=\frac {1}{h^{3/2}t^{1/3}}\frac{\tilde h^{-1/3}}{t^{1/6}}\leq \frac {1}{h^{3/2}t^{1/3}} (\frac{h}{\tilde h})^{4/3}=\frac {1}{h^{2}}(\frac ht)^{1/3} h^{1/6} (\frac{h}{\tilde h})^{4/3}.
\]
The last statement of Proposition \ref{propThmhhtilde} is obtained by summing up for all $\tilde h=2^j h$.

\section{Dispersive estimates for the wave flow in large time and the Klein Gordon flow. Proof of Theorem \ref{thmDKG} in the low frequency case}
We let $G^m_{SF}(t,x,a,y):=\sum_{j\in\mathbb{N}}G^m_{j}(t,x,a,y)$ where $G^m_j$ is defined in \eqref{eq:uhterWGMSF},
then
\begin{multline}
  \label{eq:uhterWGMSphi}
  G^m_{SF}(t,x,a,y):=\sum_{k\geq 1} \int e^{i (y\theta+t\sqrt{\lambda_k(\theta)+m^2})}   \phi(|\theta|)\phi(\sqrt{\lambda_k(\theta)})
\frac{ |\theta|^{2/3}}{ L'(\omega_{k}) } Ai(x|\theta|^{2/3}-\omega_{k}) Ai(a|\theta|^{2/3}-\omega_{k}) d\theta \,,
\end{multline}
where the integral is taken for $\theta\in\mathbb{R}^{d-1}\setminus\{0\}$. 
Let $\chi_0\in C^{\infty}_0([-2,2])$ be equal to $1$ on $[-3/2,3/2]$, let $M>1$ be sufficiently large and write $G^m_{SF}=G^m_{SF,\chi_0}+G^m_{SF,1-\chi_0}$ where for $\chi\in \{\chi_0,1-\chi_0\}$ we have set
\[
G^m_{SM,\chi}(t,x,a,y):=\sum_{k\geq 1}\int e^{i (y\theta+t\sqrt{\lambda_k(\theta)+m^2})} \phi(|\theta|)\phi(\sqrt{\lambda_k(\theta)})\chi(\frac{t\lambda_k(\theta)}{M})
e_k(x,\theta)e_k(a,\theta)d\theta.
\] 
\begin{lemma}
Let $M>1$ be large enough and $t>0$. Then there exists $C(M,d)\simeq M^{d/2}$ such that
\[
\|G^m_{SM,\chi_0}\|_{L^{\infty}(\Omega_d)}\leq C(M,d)\min\{1,\frac{1}{t^{\frac{d}{2}}}\}.
\]
\end{lemma}
\begin{proof}
Writing $\theta=\rho\Theta$ with $\rho=|\theta|$, we have
\begin{multline}
G^{m}_{SF,\chi_0}(t,x,a,y)=\sum_{k\geq 1}\int_0^{\infty}\int_{\Theta\in\mathbb{S}^{d-2}}e^{-i\rho|y|\Theta_1}\sin^{d-2} \Theta_1d\Theta \rho^{d-2}e^{i t\sqrt{\lambda_k(\rho)+m^2}} \phi(\rho)\phi(\sqrt{\lambda_k(\rho)})\chi_0(\frac{t\lambda_k(\rho)}{M})\\
\frac{ \rho^{2/3}}{ L'(\omega_{k}) } Ai(x\rho^{2/3}-\omega_{k}) Ai(a\rho^{2/3}-\omega_{k}) d\rho,
\end{multline}
where $e_1=(1,0,...,0)\in\mathbb{R}^{d-1}$. On the support of $\chi_0(\frac{t\lambda_k(\rho)}{M})(\sum_{j}\psi(2^j\rho))$ we have $\lambda_k(\rho)=\rho^2+\omega_k\rho^{4/3}\leq M/t$ which implies $\rho\leq \sqrt{M/t}$ and $\omega_k\leq \Big(M/t-\rho^2\Big)/\rho^{4/3}$ (as on the support of $\psi(2^j\rho)$, $j\geq 1$, $\rho$ doesn't vanish). Let $L=L(M/t,\rho):=\Big(M/t-\rho^2\Big)^{3/2}/\rho^{2}$; as $\omega_k\simeq k^{2/3}$, it follows that on the support of the symbol of $G^{m}_{SF,\chi_0}$ we must have $k\leq L(M/t,\rho)$. We estimate $G^{m}_{SF,\chi_0}$ as follows
\begin{multline}
|G^{m}_{SF,\chi_0}(t,\cdot)|\leq \int_{\rho\leq \sqrt{M/t}}\Big|\int_{\mathbb{S}^{d-2}}d\Theta\Big|\rho^{d-2+2/3}\Big|\sum_{1\leq k\leq L(M/t,\rho)} \frac{1}{ L'(\omega_{k})} Ai(x\rho^{2/3}-\omega_{k}) Ai(a\rho^{2/3}-\omega_{k})\Big| d\rho\\
\lesssim  \int_{\rho\leq \sqrt{M/t}}\rho^{d-2+2/3}\Big(\sum_{1\leq k\leq L(M/t,\rho)} \frac{1}{ L'(\omega_{k}) } Ai^2(x\rho^{2/3}-\omega_{k})\Big)^{1/2}\Big(\sum_{1\leq k\leq L(M/t,\rho)} \frac{1}{ L'(\omega_{k}) } Ai^2(a\rho^{2/3}-\omega_{k}) \Big)^{1/2} d\rho\\
\leq \int_{\rho\leq \sqrt{M/t}}\rho^{d-2+2/3}\Big(M/t-\rho^2\Big)^{1/2}/\rho^{2/3}d\rho=\int_{\rho\leq \sqrt{M/t}}\rho^{d-2}\Big(M/t-\rho^2\Big)^{1/2}d\rho,
\end{multline}
where in the second line we have applied the Cauchy-Schwarz inequality and then used \eqref{estairy2} from Lemma \ref{lemsob}.
Taking $\rho=\sqrt{\frac{M}{t}}w$ gives
\[
|G^{m}_{SF,\chi_0}(t,\cdot)|\lesssim \Big(\frac Mt\Big)^{\frac{d-2+1+1}{2}}\int_{w\leq 1}(1-w^2)^{1/2}dw\leq M^{d/2}/t^{d/2}.
\]
Let now $t<1$ and let $\rho=2^{-j}\tilde\rho$ for some $j\geq 0$ : as $\lambda_{k}(\theta)\leq 4$ on the support of $\phi$, then for fixed $j\geq 0$, the sum over $k$ in $G^m_j$ is finite as $\omega_k\leq 4\times 2^{4j/3}$. We estimate each $G^m_j(t,\cdot)$ as follows
\begin{multline}
  \label{eq:uhterWGM1chiSFtsmall}
  |G^m_{j}(t,x,a,y)|=\Big|\sum_{1\leq k\lesssim 2^{2j}}\int_0^{\infty}2^{-j}\int_{\Theta'\in \mathbb{R}^{d-2}} e^{-i2^{-j}\tilde\rho|y|\sqrt{1-|\Theta'|^2}}c(\Theta')d\Theta' (2^{-j}\tilde\rho)^{d-2}e^{i t\sqrt{\lambda_k(\rho)+m^2}} \\
\phi(\sqrt{\lambda_k(2^{-j}\tilde\rho)})\psi_2(\tilde\rho)\frac{(2^{-j}\tilde\rho)^{2/3}}{ L'(\omega_{k}) } Ai(x(2^{-j}\tilde\rho)^{2/3}-\omega_{k}) Ai(a(2^{-j}\tilde\rho)^{2/3}-\omega_{k}) d\tilde\rho\Big|\\
\lesssim 2^{-j(d-1+2/3)}\Big(\sum_{1\leq k\lesssim 2^{2j}} \frac{1}{ L'(\omega_{k}) } Ai^2(x\rho^{2/3}-\omega_{k})\Big)^{1/2}\Big(\sum_{1\leq k\lesssim 2^{2j}} \frac{1}{ L'(\omega_{k}) } Ai^2(a\rho^{2/3}-\omega_{k}) \Big)^{1/2}\lesssim 2^{-j(d-1)}.
\end{multline}
As the sum over $j\geq 0$ is convergent, we obtain a uniform bound for $G^m_{SF,\chi_0}(t,\cdot)$ for all $t<1$.
\end{proof}

We are left with $G^m_{SF,1-\chi_0}(t,\cdot)=\sum_{j\in\mathbb{N}}G^m_{j,1-\chi_0}(t,\cdot)$, where $G^m_{j,1-\chi_0}$ has the same form as \eqref{eq:uhterWGMSF} with the additional cut-off $(1-\chi_0)\Big(\frac{t\lambda_k(\theta)}{M}\Big)$ inserted into the symbol. Write, for some symbol $c(\Theta')$,
\begin{multline}
  \label{eq:uhterWGM1chiSF}
  G^m_{j,1-\chi_0}(t,x,a,y):=\sum_{k\geq 1}\int_0^{\infty}\int_{\Theta'\in \mathbb{R}^{d-2}} e^{-i\rho|y|\sqrt{1-|\Theta'|^2}}c(\Theta')d\Theta'\rho^{d-2}e^{i t\sqrt{\lambda_k(\rho)+m^2}} \phi(\rho)\phi(\sqrt{\lambda_k(\rho)})\\
  (1-\chi_0)(\frac{t\lambda_k(\rho)}{M})
\psi_2(2^{j}\rho)\frac{ \rho^{2/3}}{ L'(\omega_{k}) } Ai(x\rho^{2/3}-\omega_{k}) Ai(a\rho^{2/3}-\omega_{k}) d\rho.
\end{multline}

We make the change of variables $\rho=2^{-j}\tilde\rho$, with $\tilde \rho\in [\frac 34, 2]$ on the support of $\psi_2(\tilde\rho)$, and set
\begin{equation}\label{phikj}
\phi^m_{k,j}(t,x,a,y,\tilde \rho,\Theta'):=-2^{-j}\tilde\rho |y|\sqrt{1-|\Theta'|^2}+t\sqrt{m^2+2^{-2j}\tilde\rho^2+2^{-4j/3}\tilde\rho^{4/3}\omega_k}.
\end{equation}
The first order derivative of $\phi^m_{k,j}$ with respect to $\tilde\rho$ is given by
\begin{equation}\label{der1phikj}
\partial_{\tilde\rho}\phi^m_{k,j}(t,x,a,y,\tilde \rho,\Theta')=-2^{-j} |y|\sqrt{1-|\Theta'|^2}+t\frac{2^{-2j}\tilde\rho+\frac 232^{-4j/3}\tilde\rho^{1/3}\omega_k}{\sqrt{m^2+2^{-2j}\tilde\rho^2+2^{-4j/3}\tilde\rho^{4/3}\omega_k}},
\end{equation}
and the second order derivative with respect to $\tilde \rho$ is given by
\begin{multline}\label{seconderphikj}
\partial^2_{\tilde\rho}\phi^m_{k,j}(t,x,a,y,\tilde \rho,\Theta')=\frac{t}{\sqrt{m^2+2^{-2j}\tilde\rho^2+2^{-4j/3}\tilde\rho^{4/3}\omega_k}^3}\Big(m^2(2^{-2j}+\frac 29 \tilde\rho^{-2/3}2^{-4j/3}\omega_k)\\-\frac 19 (2^{-4j/3}\omega_k) 2^{-2j}\tilde\rho^{4/3}-\frac 29\tilde\rho^{2/3} (2^{-4j/3}\omega_k)^2 \Big).
\end{multline}
Notice that if $m=0$, the second order derivative of $\phi^{m=0}_{k,j}$ does not vanish anywhere (since the two terms in the second line of \eqref{seconderphikj} have same sign). On the other hand, for $m=1$ and $j\geq 0$, the second derivative of $\phi^{m=1}_{k,j}$ may cancel at some $\tilde\rho$ on the support of $\psi_2$ : this may happen for an unique $k=k(j)\sim 2^{4j/3}$. Since the phase functions may behave slightly different according to whether $m=0$ or $m=1$, we deal separately with these cases. In the following, we consider the case $m=0$ and then explain how to deal with the degenerate critical points of $\phi_{k(j),j}$.

\subsubsection{The wave flow}

Let $m=0$. We start by noticing that on the support of $1-\chi_0$ we have $t\geq \frac{3M}{2\lambda_k(\rho)}$ and that on the support of $\phi(\sqrt{\lambda_k(\rho)})$ we have $\lambda_k(\rho)\leq 4$: this implies that on the support of the symbol we must have $t\geq \frac 38 M$. We first deal with the case when $|\theta|$ is not too small. 

\begin{prop}\label{propAG}
Let $M$ be sufficiently large and $t\gtrsim M$. There exists a constant $C>0$ independent of $t$ such that for all $\epsilon>0$ 
\begin{equation}\label{sumjpetit}
\Big|\sum_{2^j\leq t^{(1-\epsilon)/4}}G^{m=0}_{j,1-\chi_0}(t,\cdot)\Big|\leq \frac{C}{|t|^{(d-1)/2}}.
\end{equation}
\end{prop}
\begin{proof}
We apply the stationary phase in $\rho=2^{-j}\tilde \rho$ as in Section \ref{secWGM} with phase function $\phi_{k,j}$ and with the Airy factors as part of the symbol. For $m=0$, the main contribution in brackets in \eqref{seconderphikj} is $-\frac 29 (2^{-4j/2}\omega_k)^2\tilde\rho^{2/3}$ (when $j\geq 2$ this is obvious; when $j\in\{0,1\}$ is small and when $2^{-2j}\simeq 2^{-4j/3}\omega_k\simeq 1$ this remains true as the last two terms in \eqref{seconderphikj} have same sign) and the second derivative behaves like
\[
\partial^2_{\tilde\rho}\phi^{m=0}_{k,j}\simeq -\frac 29 t(2^{-4j/3}\omega_k)^{1/2}.
\]
Using \eqref{derivelAirybounds} it follows that, in order to apply the stationary phase with respect to $\tilde\rho$ with the Airy factors in the symbol we must have, for some $\nu>0$,
\begin{equation}\label{tocheckkj}
(t(2^{-4j/3}\omega_k)^{1/2})^{1/2-\nu}\geq \omega_k^{3/2}.
\end{equation}
The last inequality can be re-written as $(t (2^{-4j/3}\omega_k)^{1/2})^{1/2-\nu}\geq 2^{2j}(2^{-4j/3}\omega_k)^{3/2}$,
which is equivalent to $t^{1/2-\nu}\geq 2^{2j}(2^{-4j/3}\omega_k)^{(5+2\nu)/4}$, and as $2^{-4j/3}\omega_k\leq 4$, we need to assume $2^{2j}\leq t^{1/2-\nu}$ for some $\nu>0$. For $j$ such that $2^j\leq t^{(1-\epsilon)/4}$, \eqref{tocheckkj} holds with $\nu=\epsilon/2$. The phase $\phi_{k,j}$ is stationary when $2^{-j}|y|\sim |t|\sqrt{2^{-4j/3}\omega_k}$ and on the support of the symbol we have $|t|\sqrt{2^{-4j/3}\omega_k}\geq \frac 12 |t|2^{-4j/3}\omega_k\gtrsim M$: indeed, this follows using $\sqrt{\lambda_k(\rho)}\geq  \frac 12 \lambda_k(\rho)$ (as $\sqrt{\lambda_k(\rho)}\leq 2$) and $\lambda_k(\rho)=\rho^2+\rho^{4/3}\omega_k\sim \rho^{4/3}\omega_k$. Moreover, the phase is stationary with respect to $\Theta'$ at $\Theta'=0\in \mathbb{R}^{d-2}$, and the determinant of the Hessian matrix of second order derivative is $\sim (2^{-j}|y|)^{d-2}\sim (|t|\sqrt{2^{-4j/3}\omega_k})^{d-2}$. We can therefore apply the stationary phase with respect to both $\tilde\rho$ and $\Theta'$, %
which yields
\begin{equation}\label{sumdecayjpetitwave}
|G^{m=0}_{j,1-\chi_0}(t,\cdot)|\lesssim \sum_{k\lesssim 2^{2j}} 2^{-j} \frac{2^{-j(d-2)}}{(|t| \sqrt{2^{-4j/3}\omega_k})^{(d-2)/2}}\times \frac{2^{-2j/3}}{L'(\omega_k)}\frac{1}{(t\sqrt{2^{-4j/3}\omega_k})^{1/2}}.
\end{equation}
Using that $L'(\omega_k)\simeq \sqrt{2\omega_k}$, \eqref{sumdecayjpetitwave} reads as (for $d\geq 2$)
\[
|G^{m=0,d=2}_{j,1-\chi_0}(t,\cdot)|\lesssim \frac{1}{|t|^{1/2}}\times 2^{-j(1+2/3-1/3)}\sum_{k\leq 2^{2j}}\frac{1}{\omega_k^{3/4}}\leq \frac{1}{|t|^{1/2}}2^{-4j/3}\times 2^j= \frac{2^{-j/3}}{|t|^{1/2}},
\]
\[
|G^{m=0,d= 3}_{j,1-\chi_0}(t,\cdot)|\lesssim \frac{1}{|t|}\times 2^{-2j}\log (2^{2j}),\quad |G^{m=0,d\geq 4}_{j,1-\chi_0}(t,\cdot)|\lesssim \frac{1}{|t|^{(d-1)/2}}\times 2^{-2j(2d-3)/3}.
\]
Summing up over $j$ such that $2^j\leq t^{(1-\epsilon)/4}$ allows to conclude.
\end{proof}
From now on we let $t\gtrsim M$ and $2^j\gtrsim |t|^{(1-\epsilon)/4}$, which corresponds to small initial angles $\theta$. Fix $a\in 2^{2j_0/3}[\frac 12, 2]$ for some $j_0\in\mathbb{Z}$. We first notice that for $a\rho^{2/3}> \omega_k$ the estimates become trivial using the exponential decay of the Airy function on the positive real line. Let $a\rho^{2/3}< \omega_k$. As $\rho=2^{-j}\tilde \rho$, $\tilde \rho\in[\frac 34, 2]$ on the support of $\psi_2$ and $2^{-4j/3}\omega_k\leq \lambda_k(2^{-j}\tilde \rho)\leq 4$ on the support of $\phi$, we obtain the condition $\frac {1}{2}(\frac 34)^{2/3} 2^{2(j_0-j)/3}\leq a 2^{-2j/3}\tilde\rho^{2/3}\leq \omega_k\leq 4\times 2^{4j/3}$ which further yields $2^{2(j_0-j)/3}\leq 2(4/3)^{2/3} 4\times 2^{4(j+1)/3}$, and as $(4/3)^{2/3}<2^{1/3}$ we find $j_0< 3(j+2)$. We start with the sum over $j>j_0$ as in this case the Airy factors can be dealt with using \eqref{eq:Apm} and the decay of $A_{\pm}$.  

\begin{prop}\label{proplemapetit}
There exists a constant $C=C(d)$ independent of $j_0$ or $M$, such that the following holds
\[
\Big|\sum_{j>j_0}G^{m=0}_{j,1-\chi_0}(t,\cdot)\Big|\leq \frac{C}{|t|^{(d-1)/2}}.
\]
\end{prop}
\begin{proof}
Using Proposition \ref{propAG}  
we are left with the case
 $\rho=2^{-j}\tilde \rho$ with $2^{j}\geq t^{(1-\epsilon)/4}$, which corresponds to waves that propagate 
 within directions of very small angles $\lesssim M^{-1/4}$. As $j_0-j\leq -1$, $x\leq a$ (by symmetry of the Green function) and $a2^{-2j/3}\tilde \rho^{2/3}\leq 2\times 2^{2(j_0-j)/3}2^{2/3}\leq 2<\frac 45\omega_k$ for all $k\geq 1$, we can write both Airy factors in \eqref{eq:uhterWGM1chiSF} using \eqref{eq:Apm}. We obtain four different phase functions, where $\pm_1$ and $\pm_2$ mean independent signs,

\begin{equation}\label{defphim=0kj}
\phi^{m=0,\pm_1,\pm_2}_{k,j}:=\phi^{m=0}_{k,j}\pm_1 \frac 23 (\omega_k-a2^{-2j/3}\tilde\rho^{2/3})^{3/2}\pm_2 \frac 23 (\omega_k-x2^{-2j/3}\tilde\rho^{2/3})^{3/2},
\end{equation}
whose derivatives are given by
\begin{equation}\label{der1phikjpm}
\partial_{\tilde\rho}\phi^{m=0,\pm_1,\pm_2}_{k,j}=\partial_{\tilde\rho}\phi^{m=0}_{k,j}\mp_1 \frac 23 a2^{-2j/3}\tilde\rho^{-1/3} (\omega_k-a2^{-2j/3}\tilde\rho^{2/3})^{1/2}\mp_2 \frac 23 x2^{-2j/3}\tilde\rho^{-1/3}(\omega_k-x2^{-2j/3}\tilde\rho^{2/3})^{1/2},
\end{equation}
\begin{multline}\label{derphikjpm}
\partial^2_{\tilde\rho}\phi^{m=0,\pm_1,\pm_2}_{k,j}=\partial^2_{\tilde\rho}\phi^{m=0}_{k,j}\pm_1 \frac 29 a2^{-2j/3}\tilde\rho^{-4/3} (\omega_k-a2^{-2j/3}\tilde\rho^{2/3})^{1/2}\Big(1+\frac{a2^{-2j/3}\tilde\rho^{2/3}}{(\omega_k-a2^{-2j/3}\tilde\rho^{2/3})}\Big)\\
\pm_2 \frac 29 x2^{-2j/3}\tilde\rho^{-4/3}(\omega_k-x2^{-2j/3}\tilde\rho^{2/3})^{1/2}\Big(1+\frac{x2^{-2j/3}\tilde\rho^{2/3}}{(\omega_k-x2^{-2j/3}\tilde\rho^{2/3})}\Big).
\end{multline}
The main term in each of the two brackets in \eqref{derphikjpm} is $1$. We distinguish two main regimes :
\begin{itemize}
\item If $|t|> 4a$, $x\leq a$, then the last two terms in \eqref{der1phikjpm} (corresponding to the derivatives of the phase functions the Airy factors $A_{\pm}$) are small compared to the second term in the right hand side of \eqref{der1phikj} and therefore
the phase is stationary in $\tilde \rho$ for $2^{-j}|y|\sqrt{1-|\Theta'|^2}\sim |t|(2^{-4j/3}\omega_k)^{1/2}$; for such values, the parameter $2^{-j}|y|$ of $\sqrt{1-|\Theta'|^2}$ is large ; moreover, $|\partial^2_{\tilde\rho}\phi^{m=0,\pm_1,\pm_2}_{k,j}|\sim |t|(2^{-4j/3}\omega_k)^{1/2}$ (as in this regime the second order derivatives of the Airy functions remain much smaller than $\partial^2_{\tilde\rho}\phi^{m=0}_{k,j}$) ; we conclude exactly as in the proof of Proposition \ref{propAG}, as the stationary phase applies in both $\tilde\rho$ and $\Theta'$. The estimates are even better than in \eqref{sumdecayjpetitwave} due to the decay of the Airy factors $A_{\pm}$.

\item Let $|t|\leq  4a $, then of $\partial^2_{\tilde\rho}\phi^{m=0,\pm_1,\pm_2}_{k,j}$ may be small (as
$t\sqrt{2^{-4j/3}\omega_k}\mp_1 a \sqrt{2^{-4j/3}\omega_k-a2^{-2j}}\mp_2 x \sqrt{2^{-4j/3}\omega_k-x2^{-2j}}$ may be close to $0$, in which case the (RHS) term of \eqref{derphikjpm} may be small). Moreover, we may also have $2^{-j}|y|$ small there where the phase is stationary with respect to $\tilde \rho$: when this is the case we cannot apply the stationary phase with respect to $\Theta'$. However, as for $x\leq a$ the symbols of $A_{\pm}(x2^{-2j/3}\tilde \rho^{2/3}-\omega_k)$ and $A_{\pm}(x2^{-2j/3}\tilde \rho^{2/3}-\omega_k)$ decay like $\omega_k^{-1/4}$, we bound $|G^{m=0}_{j,1-\chi_0}(t,\cdot)|$ as follows
\begin{align}\label{tpetitagrand}
|G^{m=0}_{j,1-\chi_0}(t,\cdot)|& \leq  2^{-j(1+(d-2)+2/3)}\sum_{k\lesssim 2^{2j}} \frac{1}{L'(\omega_k)\omega_k^{1/4+1/4}}\leq 2^{-j(1+(d-2)+2/3)}2^{2j/3}\\
\nonumber
&\leq 2^{-j(d-1)}\lesssim \frac{2^{-2j(d-1)/3}}{|t|^{(d-1)/2}}, \quad \text{ as } t\lesssim 4a\leq 16\times 2^{2j_0/3}<16\times 2^{2j/3},
\end{align}
where, in order to estimate the sum over $k$ we have used that $L'(\omega_k)\simeq \sqrt{2\omega_k}$ and $\omega_k\simeq k^{2/3}$.
Summing up over $j>j_0$ achieves the proof.
\end{itemize}
\end{proof}
In Proposition \ref{proplemapetit} we have considered only values $j>j_0$ so that the Airy factors could be written as in \eqref{eq:Apm}. The next lemma deals with $j\leq j_0<3(j+2)$. %
\begin{prop}\label{lemagrand}
There exists a constant $C=C(d)$ independent of $j_0$, $M$, such that the following holds
\[
\Big|\sum_{j\leq j_0< 3(j+2),  2^j\gtrsim t^{(1-\epsilon)/4}}G^{m=0}_{j,1-\chi_0}(t,\cdot)\Big|\leq \frac{C}{|t|^{(d-1)/2}}.
\]
\end{prop}
\begin{proof}
Notice that if $a\in 2^{2j_0/3}[\frac 12, 2]$ is chosen such that $2^{j_0}\leq t^{(1-\epsilon)/4}$ for some $\epsilon>0$, then Proposition \ref{propAG} applies for all $j\leq j_0$ and we conclude that $G^{m=0}_{j,1-\chi_0}(t,\cdot)$ is bounded by $C/|t|^{(d-1)/2}$. Therefore we consider $j_0$ such that $2^{j_0}\gtrsim t^{(1-\epsilon)/4}$ for any $\epsilon>0$.

For values $k$ such that $\omega_k \leq \frac 14 2^{2(j_0-j)/3}$ we obtain trivial contributions in \eqref{eq:uhterWGM1chiSF} using the decay of the Airy functions. Therefore we split the sum over $k$ in two parts, according to whether $\omega_k\sim 2^{2(j_0-j)/3}$ or $\omega_k\geq 2^{2(j_0-j+2)/3}$. In the second case, we obtain $\frac 45\omega_k\geq  a2^{-2j/3}\tilde\rho^{2/3}$ and using \eqref{eq:Apm} yields again four phase functions as in \eqref{defphim=0kj}; in this case the proof follows as in Lemma \ref{proplemapetit} for $|t|>4a$. When $|t|\leq 4a$ we have $|t|\leq 4a \leq 8\times 2^{2j_0/3}$.
As $j\leq j_0< 3(j+2)$, we write $j=[j_0/3]-1+l$, with $l\in [0,2([j_0]+1)/3]\cap\mathbb{N}$ and we obtain, as in \eqref{tpetitagrand},
\begin{multline}
\sum_{j\leq j_0< 3(j+2)} |G^{m=0}_{j,1-\chi_0}(t,\cdot)|\lesssim \sum_{j\leq j_0< 3(j+2)}2^{-j(d-1)}= \sum_{j=[j_0/3]-1+l,l\leq  2([j_0]+1)/3}2^{-([j_0]/3-1+l)(d-1)}\\
\leq \frac{4^{d-1}}{|t|^{(d-1)/2}}\sum_{l\geq 0} 2^{-l(d-1)}\leq \frac{4^{d-1}}{|t|^{(d-1)/2}}.
\end{multline}
Let $\omega_k\simeq  2^{2(j_0-j)/3}$ : we have to obtain dispersive bounds for $\sum_{j\leq j_0\leq 3(j+2)} G^{m=0,\#}_{j,1-\chi_0}(t,\cdot)$, where
\begin{multline}
  \label{eq:uhterWGM1chiSFtg}
 G^{m=0,\#}_{j,1-\chi_0}(t,x,a,y):=\sum_{k\geq 1}\int_0^{\infty}\int_{\Theta' \in\mathbb{R}^{d-2}} e^{\frac ih\rho|y|\sqrt{1-|\Theta'|^2}}c(\Theta')d\Theta'\rho^{d-2}e^{i t\sqrt{\lambda_k(\rho)}} \phi(\rho)\phi(\sqrt{\lambda_k(\rho)})\\
\tilde\psi(\frac{\omega_k}{2^{2(j_0-j)/3}})(1-\chi_0)(\frac{t\lambda_k(\rho)}{M})
\psi_2(2^{j}\rho)\frac{ \rho^{2/3}}{ L'(\omega_{k}) } Ai(x\rho^{2/3}-\omega_{k}) Ai(a\rho^{2/3}-\omega_{k}) d\rho.
\end{multline}
Here we have introduced a cut-off $\tilde\psi$ supported in $[\frac 14,4]$ and defined the "tangent" flow as the restriction of $G^{m=0}_{j,1-\chi_0}$ to the sum over $k$ such that $\omega_k\simeq 2^{2(j_0-j)/3}$.
We need the following lemma, whose proof is similar to that of Proposition \ref{propAG} :
\begin{lemma}\label{propAGtg}
Let $t\geq \frac 23M$. There exists a constant $C=C(d)>0$ such that for small $\epsilon>0$ the following holds
\begin{equation}\label{sumjpetitbis}
\Big|\sum_{2^{2(j_0-j)}\leq t^{1-\epsilon}}G^{m=0,\#}_{j,1-\chi_0}\Big|\leq \frac{C}{|t|^{(d-1)/2}}.
\end{equation}
\end{lemma}
\begin{proof}
We follow the same approach as in Proposition \ref{propAG} : as $t$ is sufficiently large and $\omega_k\simeq 2^{2(j_0-j)/3}$ is small for $j$ close to $j_0$, we can consider the Airy factors as part of the symbol. As in \eqref{tocheckkj}, this is possible as long as
\[
t\sqrt{2^{-4j/3}\omega_k}\gg \omega_k^3,
\] 
and writing $\omega_k=2^{2(j_0-j)/3}\frac{\omega_k}{2^{2(j_0-j)/3}}$, the last inequality reads as $t\gg \frac{2^{2(j_0-j)}}{\sqrt{2^{-4j/3}\omega_k}}\Big(\frac{\omega_k}{2^{2(j_0-j)/3}}\Big)^3$.
As $2^{-4j/3}\tilde\rho^{4/3}\omega_k\leq 4$, $\tilde\rho\in[\frac 32,2]$ and $\frac{\omega_k}{2^{2(j_0-j)/3}}\in [\frac 14,4]$ on the support of $\tilde \psi$ and $\psi(\tilde\rho)$, in order to apply the stationary phase with the Airy factors in the symbol it will be enough to require $2^{2(j_0-j)}\leq t^{1-\epsilon}$ for some $\epsilon>0$ (and $t\gtrsim M$). The dispersive bounds follow like in \eqref{sumdecayjpetitwave} : 
\begin{equation}\label{sumdecayjpetitwavetg}
|G^{m=0,\#}_{j,1-\chi_0}(t,\cdot)|\lesssim \sum_{k\simeq 2^{j_0-j}(\lesssim 2^{2j})} 2^{-j} \frac{2^{-j(d-2)}}{(|t| \sqrt{2^{-4j/3}\omega_k})^{(d-2)/2}}\times \frac{2^{-2j/3}}{L'(\omega_k)}\frac{1}{(t\sqrt{2^{-4j/3}\omega_k})^{1/2}}.
\end{equation}
\end{proof}
Let now $j$ such that $2^{2(j_0-j)}\geq t^{1-\epsilon}$ for some small $\epsilon>0$ and $k$ such that $\frac{\omega_k}{2^{2(j_0-j)/3}}\in [\frac 14,4]$. Replacing the Airy factors of \eqref{eq:uhterWGM1chiSFtg} by their the integral formulas \eqref{eq:bis47} yields a new phase
\begin{equation}\label{phimkmj01}
\phi^{m=0}_{k,j}+\frac{s^3}{3}-s(\omega_k-a2^{-2j/3}\tilde\rho^{2/3}) +\frac{\sigma^3}{3}-\sigma (\omega_k-x2^{-2j/3}\tilde\rho^{2/3}),
\end{equation}
whose critical points satisfy $s^2+a2^{-2j/3}\tilde\rho^{2/3}=\omega_k$ and $\sigma^2+x2^{-2j/3}\tilde\rho^{2/3}=\omega_k$ and whose critical values equal $\phi^{m=0,\pm_1,\pm_2}_{k,j}$ defined in \eqref{defphim=0kj}. If $|s|\geq \frac 98 \sqrt{\omega_k}$ or $|\sigma|\geq \frac 98\sqrt{\omega_k}$, repeated integrations by parts provide a contribution $O(\omega_k^{-n})$ for all $n\in \mathbb{N}$ : as $\omega_k\simeq 2^{2(j_0-j)/3}$ and $2^{2(j_0-j)}\geq t^{1-\epsilon}$, then $O(\omega_k^{-n})=O(|t|^{-n})$. In the following we let $|s|,|\sigma|\leq \frac 98\sqrt{\omega_k}$. If $|s|,|\sigma|\geq \frac 18\sqrt{\omega_k}$, then the  stationary phase applies in both $s,\sigma$, the critical value of the phase becomes $\phi^{m=0,\pm_1,\pm_2}_{k,j}$ and we conclude as in Proposition
\ref{proplemapetit}. Let $|s|\leq  \frac 18\sqrt{\omega_k}$ or $|\sigma|\leq \frac 18\sqrt{\omega_k}$, then $|s+\sigma|\leq \frac 54\sqrt{\omega_k}$ and $|s-\sigma|\leq \frac 54\sqrt{\omega_k}$.

Let first $|t|\geq 20 a$, then the phase is stationary in $\tilde\rho$ if $2^{-j}|y|\sqrt{1-|\Theta'|^2}\sim |t|(2^{-4j/3}\omega_k)^{1/2}$ : as a consequence, the stationary phase can applies in $\Theta'$: indeed, the phase is stationary in $\tilde \rho$ when
\begin{equation}\label{phastatrhotgrand}
2^{-j}|y|\sqrt{1-|\Theta'|^2}+t\frac{2^{-2j}\tilde\rho+\frac 232^{-4j/3}\tilde\rho^{1/3}\omega_k}{\sqrt{2^{-2j}\tilde\rho^2+2^{-4j/3}\tilde\rho^{4/3}\omega_k}}-(s+\sigma)\omega_k=0,
\end{equation}
where $|s+\sigma|\leq \frac 54\sqrt{\omega_k}$. Since $j_0$ is large (recall that $2^{j_0-j}\geq t^{(1-\epsilon)/2}$) and so is $j$ (recall that we are dealing with $j\leq j_0\leq 3(j+2)$), the middle term in \eqref{phastatrhotgrand} equals
\[
|t|\frac{2^{-2j}\tilde\rho+\frac 232^{-4j/3}\tilde\rho^{1/3}\omega_k}{\sqrt{2^{-2j}\tilde\rho^2+2^{-4j/3}\tilde\rho^{4/3}\omega_k}}=\frac 23 |t|\tilde\rho^{-1/3}\sqrt{2^{-4j/3}\omega_k}(1+\tilde\rho^{2/3}/(2^{2j/3}\omega_k)+O((2^{2j/3}\omega_k)^{-2})).
\]
As $a\in 2^{2j_0/3}[\frac 12,2]$, $2^{2(j_0-j)/3}\geq \frac 14\omega_k$ and $\tilde\rho\in[\frac 34,2]$, we have 
\begin{align*}
|t|\frac{2^{-2j}\tilde\rho+\frac 232^{-4j/3}\tilde\rho^{1/3}\omega_k}{\sqrt{2^{-2j}\tilde\rho^2+2^{-4j/3}\tilde\rho^{4/3}\omega_k}}&\geq \frac 23 \times 20 a\times 2^{-2j/3}\sqrt{\omega_k}\times \tilde\rho^{-1/3}\\
&\geq \frac 23 \times 20\times \frac 12 2^{2j_0/3} \times 2^{-2j/3}\sqrt{\omega_k}\times \tilde\rho^{-1/3}\\
&\geq \frac{20}{3\times 4\times\tilde\rho^{1/3}} \omega_k^{3/2}\geq (5/3)2^{-1/3}\omega_k^{3/2},
\end{align*}
and as $(5/3)\times 2^{-/3}\simeq 1,32>5/4$ it follows that \eqref{phastatrhotgrand} can hold only if $2^{-j}|y|\sim |t|(2^{-4j/3}\omega_k)^{1/2}\gtrsim M$, so the stationary phase applies in $\Theta'$ as the parameter is sufficiently large. Moreover, for $|t|\geq 20a$, the second order derivative in $\tilde\rho$ is $\simeq |t|(2^{-4j/3}\omega_k)^{1/2}$, so the stationary phase applies also in $\tilde\rho$ and the estimates we obtain are exactly as in \eqref{sumdecayjpetitwavetg}.

Let now $t\leq 20a$ and $|s|,|\sigma|\leq \frac 98\sqrt{\omega_k}$ with $|s+\sigma|\leq \frac 54\sqrt{\omega_k}$. Notice that in this case the stationary phase in $\Theta'$ may not apply at all as the parameter $2^{-j}|y|$ may remain small.  After the change of variable $\rho=2^{-j}\tilde\rho$, the sum over $j$ of \eqref{eq:uhterWGM1chiSFtg} can be bounded by
\begin{multline}
  \label{eq:uhterWGM1chiSFtgsumj}
\sum_{j\leq j_0\leq 3(j+2)}2^{-(d-1)j-2j/3}\int_0^{\infty}\tilde\rho^{d-2+2/3}\psi_2(\tilde\rho)\Big| \sum_{k\sim 2^{j_0-j}}
\frac{1}{ L'(\omega_{k}) } Ai(x\rho^{2/3}-\omega_{k}) Ai(a\rho^{2/3}-\omega_{k}) \Big| d\rho\\
\lesssim \sum_{j\leq j_0\leq 3(j+2)}2^{-(d-1)j-2j/3}2^{(j_0-j)/3}=2^{-j_0(d-1)/3}\sum_{j\leq j_0\leq 3(j+2)}2^{(j_0/3-j)d}
\leq
\frac{2\times 7^{d-1}}{t^{\frac{d-1}{2}}},
\end{multline}
where we have used that $k\sim 2^{j_0-j}$ on the support of $\tilde\psi(\frac{\omega_k}{2^{2(j_0-j)/3}})$ and then applied the Cauchy-Schwarz inequality followed by \eqref{estairy2} with $L\sim 2^{j_0-j}$ to obtain the first sum in the second line. To deduce the last inequality we have used $|t|\leq 20 a\leq 40\times  2^{2j_0/3}$ which gives $2^{-j_0(d-1)/3}< \frac{7^{d-1}}{|t|^{(d-1)/2}}$ together with the fact that the sum over $j$ is taken over $j_0/3-j\leq 5/3$ hence it is convergent.

\end{proof}

\subsubsection{The Klein-Gordon flow}

Let now $m=1$, then the term in brackets in \eqref{seconderphikj} is 
\begin{equation}\label{terminbrak}
f_{k,j}(\tilde\rho^{2/3}):=\frac 29(2^{-4j/3}\omega_k)\tilde\rho^{-2/3}\Big(1-(2^{-4j/3}\omega_k)\tilde\rho^{4/3}\Big) -\frac 19 (2^{-4j/3}\omega_k)2^{-2j}\tilde\rho^{4/3}+2^{-2j}.
\end{equation}
The following expansions will be useful (see \cite[(2.52), (2.64)]{AFbook})):
 $\omega_1=2.3381074105$, $\omega_4=6.7867080901$, $\omega_8=11,0085243037$ and for all $j\geq 3$, $\omega_{k=2^{2j}}=(\frac{3\pi}{8})^{2/3}\times 2^{4(j+1)/3}(1+O(2^{-2(j+1)}))$. \\

 We now separate two different situations: the first one was alluded to in the introduction and uncovers a new effect that leads to a worse decay estimate for Klein-Gordon. The second case is dealt with as we did for the wave equation.
\begin{enumerate}
\item Let $j=0$ and $k=1$ : we claim that $\partial^2_{\tilde\rho}\phi^{m=1}_{0,1}$ may vanish for some $\tilde\rho$ near $1$ for all $t$. Let $j=0$ and $k=1$ and let $z=\tilde\rho^{2/3}$ : the function $f_{1,0}(z):=1+\frac 29 z^{-1}\omega_1-\frac 19 \omega_1 z^2-\frac 29 z\omega_1^2$ satisfies $f_{1,0}(1)=1+\frac 19 \omega_1-\frac 29 \omega_1^2\simeq 0,04$ and $f_{1,0}'(z)<0$ for all $z>0$, hence $f_{0,1}$ is strictly decreasing, so the second derivative of $\phi^{m=1}_{1,0}$ does cancel for some $\tilde\rho$ very close to $1$. In the same way, for each $j\geq 1$, there exists at most one value $k(j)\geq 2$, $k(j)\sim 2^{2j}$, such that $f_{k,j}(z)$ vanishes for some $z$ near $1$. We have $2^{-4j/3}\omega_{k(j)}\sim 1$.

Let $j\geq 0$ and $k=k(j)$ and let $a\sim 2^{2j_0/3}$ for some $j_0\geq 0$. As $a2^{-2j/3}\lesssim \omega_{k(j)}\sim 2^{4j/3}$, we must have $j_0\leq 3j+c_0$ for some fixed $c_0$ depending only on the support of the $\psi_2$, $\phi$. If $t$ is large but such that $t\leq M_1 a$ for some $M_1>2$, then $2^{-2j_0/3}\leq M_1/t$ %
and we can proceed as in \eqref{eq:uhterWGM1chiSFtgsumj} :   the change of variables $\rho=2^{-j}\tilde \rho$ yields a factor $2^{-(d-1)j-2j/3}$ as in \eqref{eq:uhterWGM1chiSFtgsumj}. Using \eqref{estairy2}, the sum over $k\sim 2^{2j}$ of the Airy factors yields in turn a factor $2^{2j/3}$. Writing 
\[
2^{-(d-1)j}= 2^{-(d-1)(j-j_0/3)}\times 2^{-(d-1)j_0/3}\lesssim 2^{-(d-1)(j-j_0/3)} t^{-(d-1)/2},
\]
and using that the sum over $j$ is taken for $j>j_0/3$, provide bounds like $t^{-(d-1)/2}$ for the sum over $j, k(j)$. We are left with the case $t>M_1a$. Writing the Airy factors under their integral form yields phase functions as in \eqref{phimkmj01}, where $\phi^{m=0}_{k,j}$ is replaced by $\phi^{m=1}_{k,j}$ and where $k=k(j)$. The first order derivative of this new phase with respect to $\tilde\rho$ equals $\partial_{\tilde\rho}\phi^{m=1}_{k(j),j}+\frac 23 as2^{-2j/3}\tilde \rho^{-1/3}+\frac 23 a\sigma 2^{-2j/3}\tilde \rho^{-1/3}$ and using  $s^2,\sigma^2\leq 2\omega_{k(j)}\lesssim 2^{4j/3}$ it follows that the absolute value of last two terms in this derivative are bounded by $c_1 a$ and $c_1 x$ for some fixed constant $c_1>1$ ; moreover, we recall that, by symmetry of the Green function (and its spectral localizations), we can assume $x\leq a$. The first order derivative of $\phi^{m=1}_{k,j}$ is given in \eqref{der1phikj} and for $k=k(j)$, the factor of $t$ is $\sim 2^{-4j/3}\omega_k\sim 1$. Therefore, if $M_1$ is chosen sufficiently large, then the phase may be stationary in $\tilde\rho$ when $2^{-j}|y|\sim t$ ; in particular, as $t$ is large, the stationary phase with respect to $\theta$ applies. Near $\tilde \rho$ such that $\partial^2_{\tilde\rho}\phi^{m=1}_{k(j),j}=0$ we have $\partial^3_{\tilde\rho}\phi^{m=1}_{k(j),j}\sim t$ and using again that $t>M_1a\geq M_1x$ with $M_1$ large implies that the third order derivative behaves like $t$. Applying Van der Corput lemma yields a factor $t^{-1/3}$. For each $j$ we have $\frac{1}{L'(\omega_{k(j)})}\sim\frac{1}{\sqrt{\omega_{k(j)}}}\sim 2^{-2j/3}$. Eventually, the sum over $j$ and $k\sim 2^{2j}$ may be bounded by $C t^{-(d-2)/2-1/3}$ for some uniform constant $C>0$.

Notice that we cannot do better than that for $m=1$ : using \eqref{eq:AiryPoissonBis} yields a sum over $N\sim t$ and as $t$ is large, the number of waves which provide important contributions is proportional to $t$, which yields a loss worse than the one obtained using the gallery modes.\\

\item For $j\geq 0$ and $k\neq k(j)$, then $\partial^2_{\tilde\rho}\phi^{m=1}_{j,k}\neq 0$ on the support of $\psi_2$. This situation can be dealt with exactly as in the case of the wave flow and provide the same kind of bounds.
\end{enumerate}

\def\cprime{$'$} \def\cprime{$'$}

\end{document}